\documentclass[11pt]{amsart}
\usepackage{amssymb}

\makeatletter
\newcommand{\sumprime}{\if@display\sideset{}{'}\sum%
            \else\sum'\fi}
\makeatother

\begin{document}

\numberwithin{equation}{section}

\newtheorem{theorem}{Theorem}[section]
\newtheorem{proposition}[theorem]{Proposition}
\newtheorem{conjecture}[theorem]{Conjecture}
\def\theconjecture{\unskip}
\newtheorem{corollary}[theorem]{Corollary}
\newtheorem{lemma}[theorem]{Lemma}
\newtheorem{observation}[theorem]{Observation}
\newtheorem{definition}{Definition}
\numberwithin{definition}{section} 
\newtheorem{remark}{Remark}
\def\theremark{\unskip}
\newtheorem{question}{Question}
\def\thequestion{\unskip}
\newtheorem{example}{Example}
\def\theexample{\unskip}
\newtheorem{problem}{Problem}

\def\vvv{\ensuremath{\mid\!\mid\!\mid}}
\def\intprod{\mathbin{\lr54}}
\def\reals{{\mathbb R}}
\def\integers{{\mathbb Z}}
\def\N{{\mathbb N}}
\def\complex{{\mathbb C}\/}
\def\dist{\operatorname{dist}\,}
\def\spec{\operatorname{spec}\,}
\def\interior{\operatorname{int}\,}
\def\trace{\operatorname{tr}\,}
\def\cl{\operatorname{cl}\,}
\def\essspec{\operatorname{esspec}\,}
\def\range{\operatorname{\mathcal R}\,}
\def\kernel{\operatorname{\mathcal N}\,}
\def\dom{\operatorname{Dom}\,}
\def\linearspan{\operatorname{span}\,}
\def\lip{\operatorname{Lip}\,}
\def\sgn{\operatorname{sgn}\,}
\def\Z{ {\mathbb Z} }
\def\e{\varepsilon}
\def\p{\partial}
\def\rp{{ ^{-1} }}
\def\Re{\operatorname{Re\,} }
\def\Im{\operatorname{Im\,} }
\def\dbarb{\bar\partial_b}
\def\eps{\varepsilon}
\def\O{\Omega}
\def\Lip{\operatorname{Lip\,}}

\def\Hs{{\mathcal H}}
\def\E{{\mathcal E}}
\def\scriptu{{\mathcal U}}
\def\scriptr{{\mathcal R}}
\def\scripta{{\mathcal A}}
\def\scriptc{{\mathcal C}}
\def\scriptd{{\mathcal D}}
\def\scripti{{\mathcal I}}
\def\scriptk{{\mathcal K}}
\def\scripth{{\mathcal H}}
\def\scriptm{{\mathcal M}}
\def\scriptn{{\mathcal N}}
\def\scripte{{\mathcal E}}
\def\scriptt{{\mathcal T}}
\def\scriptr{{\mathcal R}}
\def\scripts{{\mathcal S}}
\def\scriptb{{\mathcal B}}
\def\scriptf{{\mathcal F}}
\def\scriptg{{\mathcal G}}
\def\scriptl{{\mathcal L}}
\def\scripto{{\mathfrak o}}
\def\scriptv{{\mathcal V}}
\def\frakg{{\mathfrak g}}
\def\frakG{{\mathfrak G}}

\def\ov{\overline}

\thanks{Supported by NSF grant 11771089}

\address{School of Mathematical Sciences, Fudan University, Shanghai, 200433, China}
 \email{boychen@fudan.edu.cn}

\title{Hardy-Sobolev type inequalities and their applications}
\author{Bo-Yong Chen}
\date{}
\maketitle

\bigskip

\begin{abstract}
This paper is devoted to various applications of Hardy-Sobolev type inequalities. We derive a new $L^2$ estimate for the $\bar{\partial}-$equation on ${\mathbb C}^n$ which yields a quantitative generalization of the Hartogs extension theorem to the case when the singularity set is not necessary compact. A $\partial\bar{\partial}-$proof of the Hartogs extension theorem  for pluriharmonic functions is also given. We show that for any negative subharmonic function $\psi$ on ${\mathbb R}^n$, $n>2$, the BMO norm of $\log |\psi|$ is bounded above by $2\sqrt{n-2}$ and $|\psi|^\gamma$ satisfies a reverse H\"older inequality for every $0<\gamma<1$. We also show that  every plurisubharmonic function is locally BMO. Several Liouville theorems for subharmonic functions on complete Riemannian manifolds are given. As a consequence, we get a Margulis type theorem that if a bounded domain in ${\mathbb C}^n$ covers a Zariski open set in a projective algebraic variety, then the group of deck transformations of the covering has trivial center.
\end{abstract}

\section{Introduction}

Many problems in analysis can be reduced to integral inequalities. This paper is devoted to the inequality
 \begin{equation}\label{eq:Hardy-Sobolev_1}
 \left[\int_{{\mathbb R}^n} |\phi|^\alpha \omega\right]^{1/\alpha}\le \left[\int_{{\mathbb R}^n} |\nabla \phi|^2 \omega'\right]^{1/2},\ \ \ \phi\in C_0^\infty({\mathbb R}^n)
 \end{equation}
 where $\omega,\omega'$ are weighted functions, i.e. measurable and positive a.e. on ${\mathbb R}^n$.
  If $\alpha=2$, $\omega=\frac{(n-2)^2}4|x|^{-2}$ and $\omega'=1$, then we have Hardy's inequality
 \begin{equation}\label{eq:Hardy-Sobolev_2}
 \frac{(n-2)^2}4\int_{{\mathbb R}^n} |\phi|^2 |x|^{-2} \le  \int_{{\mathbb R}^n} |\nabla \phi|^2.
 \end{equation}
 If $\alpha=\frac{2n}{n-2}$, $\omega=1$ and $\omega'=C_n\gg 1$, then we have Sobolev's inequality
  \begin{equation}\label{eq:Hardy-Sobolev_3}
 \left[\int_{{\mathbb R}^n} |\phi|^{\frac{2n}{n-2}} \right]^{\frac{n-2}{n}}\le C_n \int_{{\mathbb R}^n} |\nabla \phi|^2.
 \end{equation}
Thus it is reasonable to call (\ref{eq:Hardy-Sobolev_1}) a\/ {\it Hardy-Sobolev\/} type inequality. A large literature exists for Hardy-Sobolev type inequalities and their applications (see \cite{Mazja}, \cite{OpicKufner}).

We begin with a new application of (\ref{eq:Hardy-Sobolev_1}) in function theory of several complex variables.
 Let $L^2_{(p,q)}({\mathbb C}^n)$ denote the space of $(p,q)$ forms $u$ in ${\mathbb C}^n$,
$$
u={\sum}'_{|I|=p}{\sum}'_{|J|=q} u_{I,J} dz^I\wedge d\bar{z}^J
$$
where $u_{I,J}\in L^2({\mathbb C}^n)$ and $\sum'$ means that the summation is preformed only over strictly increasing multi-indices. For $u,v\in L^2_{(p,q)}({\mathbb C}^n)$ we set
$$
(u,v)={\sum}'_{I,J}\int_{{\mathbb C}^n} u_{I,J}\cdot \overline{v_{I,J}} \ \ \ {\rm and\ \ \ } \|u\|_2=\sqrt{(u,u)}.
$$
For a weight function $\omega$ and $\alpha\ge 1$, we define
$$
\|u\|^2_{\omega,\alpha}  = {\sum}'_{I,J}  \left[\int_{{\mathbb C}^{n}}  |u_{I,J}|^{\alpha}\omega\right]^{2/\alpha}.
$$

\begin{theorem}\label{th:L2Estimate}
Suppose we have an inequality
\begin{equation}\label{eq:Hardy-Sobolev_4}
 \left[\int_{{\mathbb C}^{n}}|\phi|^{\alpha} \omega\right]^{1/\alpha}\le \left[\int_{{\mathbb C}^{n}} |\nabla \phi|^2\right]^{1/2}
\end{equation}
for all real-valued functions $\phi\in C^\infty_0({\mathbb C}^{n})$. Then for any $\bar{\partial}-$closed form $v\in L^2_{(p,q)}({\mathbb C}^n)$ with
    $
 \|v\|_{\omega^{-\alpha'/\alpha},\alpha'}<\infty,
  $
  where $\alpha'$ is the dual exponent of $\alpha$, i.e. $\frac1{\alpha}+\frac1{\alpha'}=1$,
  there exists a solution to $\bar{\partial} u =v$ which satisfies the estimate
   $$
   \|u\|_2 \le  2\sqrt{2}\, \|v\|_{\omega^{-\alpha'/\alpha},\alpha'}.
   $$
   Moreover, if $n>1$, $p=0$, $q=1$ and\/ ${\mathbb C}^n\backslash {\rm supp\,}v$ contains a holomorphic cylinder, then $u=0$ on the unbounded component of\/ ${\mathbb C}^n\backslash {\rm supp\,}v$.
\end{theorem}

By a holomorphic cylinder we mean a biholomorphic image of ${\mathbb C}\times {\mathbb D}^{n-1}$ where ${\mathbb D}$ is the unit disc in ${\mathbb C}$. As a consequence, we obtain a quantitative generalization of the famous Hartogs extension theorem.

\begin{corollary}\label{Cor:Hartogs}
Let $\Omega$ be a domain in ${\mathbb C}^n$, $n>1$. Let $E$ be a closed set in $\Omega$ which satisfies the following properties:
\begin{enumerate}
\item $\Omega\backslash E$ is connected$\,;$
\item there exists $r>0$ such that $
E_{r}:=\{z\in {\mathbb C}^n:d(z,E)\le r\}\subset \Omega;
$
\item ${\mathbb C}^n\backslash E_r$ contains a holomorphic cylinder.
\end{enumerate}
 If $f$ is a holomorphic function on $\Omega\backslash E$  which satisfies either
 $$
f\in L^2\cap  L^{2n/(n+1)}(E_r\backslash E_{r/2})
$$
 or
 $$
\int_{E_r\backslash E_{r/2}} |f|^2 |z|^2<\infty,
$$
then there is a holomorphic function $F$ on $\Omega$ such that $F|_{\Omega\backslash E}=f$.
\end{corollary}

By solving the $\partial\bar{\partial}-$equation instead of the $\bar{\partial}-$equation, we may prove similarly the Hartogs extension theorem for pluriharmonic functions\footnote{The author is unable to find a proof in the literature.}.

\begin{theorem}\label{Thm:Hartogs_PH}
Let $\Omega$ be a domain in ${\mathbb C}^n$, $n>1$. Let $E$ be a compact set in $\Omega$ such that  $\Omega\backslash E$ is connected.
 If $f$ is a pluriharmonic function on $\Omega\backslash E$,  then there is a pluriharmonic function $F$ on $\Omega$ such that $F|_{\Omega\backslash E}=f$.
\end{theorem}

 The special case when $\alpha=2$ in (\ref{eq:Hardy-Sobolev_1}) is of great importance, since it links to the (spectral) theory of Schr\"odinger operators. Actually there is a question posed by Fefferman \cite{Fefferman} on when the following inequality is true
 \begin{equation}\label{eq:Hardy-Sobolev_5}
 \int_{{\mathbb R}^n} |\phi|^2 \omega \le C \int_{{\mathbb R}^n} |\nabla \phi|^2.
 \end{equation}
 In \cite{Fefferman}, a sufficient condition is given as follows
 \begin{equation}\label{eq:FeffermanCondition}
  \left[\frac1{|B|}\int_B |\omega|^\alpha\right]^{1/\alpha}\le C |B|^{-2/n}
 \end{equation}
 for all balls $B\subset {\mathbb R}^n$ and some $\alpha>1$, $C>0$ (see also \cite{Chang} for further results).

 In this paper, we give two related inequalities. Let $M$ be a Riemannian manifold with a metric given locally by $ds^2=g_{ij}dx^i dx^j$ where the summation is on the repeated indices. Let $(g^{ij})=(g_{ij})^{-1}$.  For a function $\phi$, the gradient $\nabla$ acts on $\phi$ by $(\nabla \phi)^i = g^{ij}\partial \phi/\partial x^j$. The Laplace operator associated to $ds^2$ is defined by
 $$
 \Delta =\frac1{\sqrt{{\rm det}(g_{ij})}}\frac{\partial}{\partial x^i}\left[\sqrt{{\rm det}(g_{ij})} g^{ij}\frac{\partial}{\partial x^j}\right].
 $$
  Let $d\mu=\sqrt{{\rm det}(g_{ij})} dx^1\cdots dx^n$ denote the Riemannian volume element.

  \begin{proposition}\label{prop:Integ_Ineq_mfd}
   Let $\psi$ be a $C^2$ function on $M$. Let $\eta:{\mathbb R}\rightarrow (0,\infty)$ be a $C^1$ function with $\eta'>0$. Suppose either $\phi\in C^1_0(M)$ or $\psi\in C^2_0(M)$. Then we have
    \begin{equation}\label{eq:Laplace-}
   \int_M {|\phi|^2}\left[\frac{2\Delta\psi}{\eta (-\psi)} + \frac{\eta'(-\psi)}{\eta^2(-\psi)}{|\nabla\psi|^2}\right] d\mu\le 4 \int_M \frac{|\nabla\phi|^2}{\eta'(-\psi)} d\mu
  \end{equation}
   \begin{equation}\label{eq:Laplace+}
   \int_M |\phi|^2 \left[ 2\eta (\psi)\Delta\psi + \eta'(\psi){|\nabla\psi|^2}\right] d\mu\le 4 \int_M \frac{\eta^2(\psi)}{\eta'(\psi)} |\nabla\phi|^2 d\mu.
  \end{equation}
 \end{proposition}

Although Proposition \ref{prop:Integ_Ineq_mfd} is only a fairly straightforward consequence of Green's formula, it has surprising applications in the theory of subharmonic functions and plurisubharmonic (psh) functions. Throughout this paper, we assume that subharmonic functions are not identically $-\infty$. It is well-known that any subharmonic function $\psi$ on a domain in ${\mathbb R}^n$ is $L^1_{\rm loc}$.  Actually one has $\psi\in L^\alpha_{\rm loc}$ for any $\alpha<\frac{n}{n-2}$ (cf. \cite{HormanderConvexity}, Theorem 3.2.13).

 \begin{theorem}\label{th:BMO_Subharmonic}
 If $\psi$ is a negative subharmonic function on ${\mathbb R}^n$, $n>2$, then
 \begin{enumerate}
 \item $\log |\psi|\in {\rm BMO}({\mathbb R}^n)$ and its norm satisfies
  \begin{equation}\label{eq:BMO_Bound}
 \|\log |\psi|\|_{\rm BMO}< 2\sqrt{n-2};
 \end{equation}
 \item the following reverse H\"older inequality holds: for any\/ $0<\gamma<1$, one has
 \begin{equation}\label{eq:StrongHolder}
 \left[\frac1{|B|}\int_B |\psi|^{\frac{\gamma n}{n-2}}\right]^{\frac{n-2}n}\le  \frac{C_{n}}{(1-\gamma)^2}\frac{1}{|B|}\int_{B} |\psi|^\gamma
 \end{equation}
  for all balls $B\subset {\mathbb R}^n$. Here  $C_n>0$ is a constant depending only on $n$.
 \end{enumerate}
 \end{theorem}

 Recall that a function $f\in L^1_{\rm loc}({\mathbb R}^n)$ is of BMO (bounded mean oscillation) if
 \begin{equation}\label{eq:BMO}
 \|f\|_{\rm BMO}:=\sup_B \left\{\frac1{|B|}\int_B |f-f_B|\right\}<\infty,
 \end{equation}
 where the supremum is taken over all balls $B\subset {\mathbb R}^n$ and $f_B=|B|^{-1}\int_B f$ denotes the mean value of $f$ over the ball $B$.
 BMO was first introduced by John-Nirenberg \cite{JohnNirenberg} in connection with PDE, then was of great importance in harmonic analysis since Fefferman proved that the space BMO is the dual of the real-variable Hardy space $H^1$ (cf. \cite{FeffermanStein}).  It is well-known that
 $
 L^\infty \subset {\rm BMO}\subset L^\alpha_{\rm loc}
 $
 for all $0< \alpha<\infty$.
 A famous unbounded example of BMO is $\log |x|$. Further examples and properties of BMO can be found in Stein's well-known book \cite{SteinHarmonicBook}. In particular, if a positive function $f$ satisfies a reverse H\"older inequality, then $\log f\in {\rm BMO}$.

 The fact $\log|\psi|\in {\rm BMO}$ was already contained (at least implicitly) in Dahlberg's lecture notes \cite{Dahlber}, the new contribution here is the precise bound (\ref{eq:BMO_Bound}). As $-|x|^{2-n}$ is subharmonic on ${\mathbb R}^n$ for $n>2$, it follows from (\ref{eq:BMO_Bound}) that
 \begin{equation}\label{eq:BMO_Bound_2}
 \|\log |x|\|_{\rm BMO} \le \frac2{\sqrt{n-2}}.
 \end{equation}
 Recently, it was proved by Nazarov-Sodin-Volberg \cite{Nazarov} a dimension-free estimate
 \begin{equation}\label{eq:BMO_Bound_3}
 \|\log |P|\|_{\rm BMO}\le \frac{4+\log 4}2\, {\rm deg\,}P
 \end{equation}
 for any real polynomial $P$ on ${\mathbb R}^n$. In an attempt to get a dimension-free version of John-Nirenberg's inequality, Cwikel-Sagher-Shvartsman asked whether the left side of (\ref{eq:BMO_Bound_3}) can be bounded below by $c\, {\rm deg\,}P$ for some absolute constant $c>0$ (cf. \cite{Cwikel}, p.\,133). The answer turns out to be negative by the estimate (\ref{eq:BMO_Bound_2}).

  Let $\Omega$ be a domain in ${\mathbb R}^n$. Replacing ${\mathbb R}^n$ by $\Omega$, we may define $f\in {\rm BMO}(\Omega)$ analogously.   A function $f$ is said to be of locally BMO on  $\Omega$, i.e. $f\in {\rm BMO}(\Omega,{\rm loc})$, if $f\in {\rm BMO}(\Omega')$ for every  domain  $\Omega'\subset\subset \Omega$.

  \begin{theorem}\label{th:BMO_PSH}
  If $\psi$ is a psh function on a domain $\Omega\subset {\mathbb C}^n$, then $\psi\in {\rm BMO}(\Omega,{\rm loc})$.
  \end{theorem}

 \begin{remark}
  The famous John-Nirenberg lemma \cite{JohnNirenberg} implies that there exists a positive constant $C_n$ depending only on $n$ such that for every  domain  $\Omega'\subset\subset \Omega$ one has
  $$
  e^{-\frac{C_n \psi}{\|\psi\|_{{\rm BMO}(\Omega')}}}\in L^1_{\rm loc}(\Omega').
  $$
  It follows that
  $$
 c_{\Omega'}(\psi)\ge \frac{C_n}{\|\psi\|_{{\rm BMO}(\Omega')}}
  $$
  where $c_{\Omega'}(\psi)$ is the complex singularity exponent introduced by Demailly-Koll\'ar\/ \cite{DemaillyKollar}$:$
  $$
  c_{\Omega'}(\psi):=\sup\{c\ge 0: e^{-c\psi}\in L^1_{\rm loc}(\Omega')\}.
  $$
     \end{remark}

Proposition \ref{prop:Integ_Ineq_mfd} also yields a Liouville theorem for subharmonic functions on complete Riemannian manifolds.

 \begin{theorem}\label{th:Liouville_1}
 Let $M$ be a complete Riemannian manifold and $\psi$ a continuous subharmonic function on $M$. Let $\lambda$ be a positive continuous function on $(0,\infty)$. Fix a point $x_0\in M$ and set
 $$
 v_\lambda(r):=\int_{B(x_0,r)} \lambda(|\psi|) d\mu.
 $$
 Suppose that
 \begin{enumerate}
  \item either $\int_0^1 \frac{ds}{\lambda(s)}<\infty$ and $\psi< 0$,
  \item or $\int_1^\infty \frac{ds}{\lambda(s)}<\infty$ and $\psi> 0$.
 \end{enumerate}
 Then the condition $\int_{r_0}^\infty \frac{rdr}{v_\lambda(r)}=\infty$ for some $r_0>0$ implies that $\psi\equiv {\rm const.}$
 \end{theorem}

  The study of Liouville theorems for subharmonic functions on complete Riemannian manifolds was initiated by Greene-Wu \cite{GreeneWu74}, Cheng-Yau \cite{ChengYau} and Yau \cite{Yau}. In particular, Yau  \cite{Yau} proved that there does not exist any nonconstant $L^\alpha$ subharmonic function on a complete Riemannian manifold when $\alpha>1$. On the other hand, there exists a complete Riemannian manifold which carries a non-constant positive harmonic function $\psi\in L^1(M)$ $($see e.g. \cite{LiSchoen}$)$. The theorems of Cheng-Yau \cite{ChengYau} and Yau \cite{Yau} were generalized by Karp \cite{Karp}, Li-Schoen \cite{LiSchoen} and Sturm \cite{Sturm}. In particular, Sturm \cite{Sturm} proved the important special case of Theorem \ref{th:Liouville_1} when $\lambda(t)=t^\alpha$, $\alpha\neq 1$. It was then asked by Grigor'yan if Sturm's theorem can be extended to fill the gap between $L^1$ and $L^\alpha$, $\alpha\neq 1$ (cf. \cite{Grigoryan}, p.\,60). Theorem \ref{th:Liouville_1} gives a positive answer to his question.

  Sometimes the following version of Liouville theorem is more useful.

   \begin{theorem}\label{th:Liouville_2}
 Let $M$ be a complete Riemannian manifold and $\rho$ the distance function from a fixed point $x_0\in M$.
 Let $\lambda,\kappa$ be two $C^1$  positive increasing functions on $(0,\infty)$ with
 \begin{enumerate}
 \item \ \ \ $\int_{1}^\infty \frac{dr}{\lambda(r)}<\infty$,
 \item \ \ \ $\int_{1}^\infty \frac{rdr}{ \lambda(\kappa(r))V_r}=\infty$
 \end{enumerate}
 where $V_r=|B(x_0,r)|$.
 If $\psi$ is a continuous subharmonic function on $M$ which satisfies
 \begin{equation}\label{eq:Growth-condition}
 \psi\le \kappa(\rho)
 \end{equation}
 for $\rho\gg 1$,
 then $\psi$ is a constant.
 \end{theorem}

 \begin{remark}
 The continuity hypothesis of $\psi$ can be relaxed to that $e^\psi$ is continuous. To see this, simply apply Theorem \ref{th:Liouville_2} to $\psi_N=\max\{\psi,-N\}$ then let $N\rightarrow \infty$.
 \end{remark}

  The condition (\ref{eq:Growth-condition})  is satisfied, for example, by either of the following conditions:
 \begin{enumerate}
 \item $M$ is of finite volume, and
 $$
\lambda(t)=t (\log_{k} t)^{1+\varepsilon/2} \prod_{j=1}^{k-1} \log_j t\ \ \ {\rm and\ \ \ } \kappa(t)=t^2 (\log_{k} t)^{-\varepsilon}
$$
   for all large $t$ and some $\varepsilon>0$, where
$$
\log_k t :=\overbrace{\log\log\cdots\log }^{k\,{\rm times}} t.
$$
     \item $V_r\le {\rm const.} r^2$ for all large $r$, and
  $$
\lambda(t)=t (\log_{k} t)^{1+\varepsilon/2} \prod_{j=1}^{k-1} \log_j t\ \ \ {\rm and\ \ \ } \kappa(t)=(\log t) (\log_{k+1}t)^{-\varepsilon}
$$
    for all large $t$ and some $\varepsilon>0$.
  \end{enumerate}

  As an unexpected application of Theorem \ref{th:Liouville_2}, we obtain a Margulis-type theorem as follows.
   \begin{theorem}\label{th:Zariski}
  If a bounded domain in ${\mathbb C}^n$ covers a Zariski open subset of a projective algebraic variety, then the group of deck transformations of the covering has trivial center.
 \end{theorem}
  Originally, Margulis proved that if a bounded domain in ${\mathbb C}^n$ covers a compact complex manifold  then the group of deck transformations of the covering has trivial center (cf. \cite{Lin}). His proof is completely different from the proof given here. It should be remarked that besides bounded symmetric domains there are very few bounded domains which can cover a compact complex manifold. On the other hand, there are plenty of bounded domains which can cover a Zariski open subset of a projective algebraic variety (see \cite{Griffiths}).

  This paper is organized as follows.  Theorem \ref{th:L2Estimate} and Corollary \ref{Cor:Hartogs} are proved in Section 2 and Theorem \ref{Thm:Hartogs_PH} in Section 3. In Section 4, we give some Hardy-Sobolev type inequalities. Theorem \ref{th:BMO_Subharmonic} is proved in Section 5 and Theorem \ref{th:BMO_PSH} in Section 6. Theorem \ref{th:Liouville_1} and Theorem \ref{th:Liouville_2} are proved in Section 7 and Theorem \ref{th:Zariski} in Section 8. In Section 9, we give two additional applications of Hardy-Sobolev type inequalities to singularity theory of real-analytic functions and unique continuation properties of Schr\"odinger operators.

\section{$\bar{\partial}-$equation in ${\mathbb C}^n$ and applications}

\begin{proof}[Proof of Theorem \ref{th:L2Estimate}]
Note that for any complex-valued function $\phi=\phi_1+i\phi_2\in C_0^\infty({\mathbb C}^n)$,
\begin{eqnarray}\label{eq:Hardy_5}
  && \left[\int_{{\mathbb C}^{n}} |\phi_1|^{\alpha}\omega\right]^{2/\alpha}+ \left[\int_{{\mathbb C}^{n}}  |\phi_2|^\alpha \omega \right]^{2/\alpha}\nonumber\\
   & \le & \int_{{\mathbb C}^{n}} (|\nabla \phi_1|^2+|\nabla \phi_2|^2)\nonumber\\
 & = & -\int_{{\mathbb C}^{n}} (\phi_1 \Delta \phi_1 + \phi_2 \Delta \phi_2)\nonumber\\
 & = & - \int_{{\mathbb C}^{n}} \phi \cdot \overline{\Delta \phi}
\end{eqnarray}
since we have
$$
 \int_{{\mathbb C}^{n}} \phi_1 \Delta \phi_2= \int_{{\mathbb C}^{n}} \phi_2 \Delta \phi_1
$$
in view of Green's formula.

Let $D_{(p,q)}$ denote the set of smooth $(p,q)$ forms with compact support in ${\mathbb C}^n$. It follows from the Minkowski inequality and (\ref{eq:Hardy_5}) that for any $u\in D_{(p,q)}$ \begin{eqnarray}\label{eq:BasicIneq_0}
 \|u\|^2_{\omega,\alpha}
 & \le & 2 {\sum}'_{I,J} \left\{ \left[\int_{{\mathbb C}^{n}} |{\rm Re\,}u_{I,J}|^{\alpha} \omega\right]^{2/\alpha}+\left[\int_{{\mathbb C}^{n}} |{\rm Im\,}u_{I,J}|^{\alpha} \omega\right]^{2/\alpha}\right\}\nonumber \\
  & \le & - 2 {\sum}'_{I,J}  \int_{{\mathbb C}^{n}} u_{I,J} \cdot \overline{\Delta u_{I,J}}.
\end{eqnarray}
Let $\vartheta$ denote the formal adjoint of $\bar{\partial}: D_{(p,q-1)}\rightarrow D_{(p,q)}$ under the paring $(\cdot,\cdot)$. The complex Laplacian is then given by $\Box=\bar{\partial}\vartheta+\vartheta\bar{\partial}$. It is well-known that for any $u\in D_{(p,q)}$,
$$
\Box u=-\frac14 {\sum}'_{I,J} \Delta u_{I,J} dz^I\wedge d\bar{z}^J
$$
(see e.g. \cite{ChenShaw}, P. 65). Thus (\ref{eq:BasicIneq_0}) implies the following basic inequality
\begin{equation}\label{eq:BasicIneq}
 \|u\|^2_{\omega,\alpha}  \le 8 (u,\Box u)=8 (\|\bar{\partial}u\|^2_2+\|\vartheta u\|^2_2),\ \ \ u\in D_{(p,q)}.
\end{equation}
The operator $\bar{\partial}$ defines a linear, closed, densely defined operator
 $$
L^2_{(p,q-1)}({\mathbb C}^n)\rightarrow L^2_{(p,q)}({\mathbb C}^n),
 $$
 which is still denoted by the same symbol.
Let $\bar{\partial}^\ast$ be the adjoint of $\bar{\partial}$. Since the Euclidean metric is a complete K\"ahler metric on ${\mathbb C}^n$, it is known for the standard density argument (cf. \cite{AndreottiVesentini} or \cite{HormanderBook}) that $D_{(p,q)}$ lies dense in ${\rm Dom\,}\bar{\partial}\cap {\rm Dom\,}\bar{\partial}^\ast$ for the graph norm
  $$
  u\rightarrow \|u\|_2+\|\bar{\partial} u\|_2+\|\bar{\partial}^\ast u\|_2.
  $$
  Suppose $u\in {\rm Dom\,}\bar{\partial}\cap {\rm Dom\,}\bar{\partial}^\ast$ and $D_{(p,q)}\ni u_j\rightarrow u$ in the graph norm. It follows that
  \begin{eqnarray}\label{eq:BasicIneq_2}
 \|u\|^2_{\omega,\alpha} & \le & \liminf_{j\rightarrow \infty}  \|u_j\|^2_{\omega,\alpha}\nonumber\\
 & \le & 8\,\liminf_{j\rightarrow \infty} (\|\bar{\partial}u_j\|^2_2+\|\bar{\partial}^\ast u_j\|^2_2)\nonumber\\
 & = & 8\,(\|\bar{\partial}u\|^2_2+\|\bar{\partial}^\ast u\|^2_2).
\end{eqnarray}
 Now we apply the standard duality argument. Consider the linear functional
    $$
   T:\bar{\partial}^\ast w \mapsto (w,v), \ \ \ w\in {\rm Range\,}\bar{\partial}^\ast \cap {\rm Ker\,}\bar{\partial}.
    $$
    By H\"older's inequality, we have
    \begin{eqnarray}\label{eq:BoundedFunctional}
    |(w,v)|
       & \le &  \|w\|_{\omega,\alpha}\|v\|_{\omega^{-\alpha'/\alpha},\alpha'}\nonumber\\
    & \le & 2\sqrt{2} \|\bar{\partial}^\ast w\|_2 \|v\|_{\omega^{-\alpha'/\alpha},\alpha'},
    \end{eqnarray}
    which implies that $T$ is a well-defined continuous functional on ${\rm Range\,}\bar{\partial}^\ast \cap {\rm Ker\,}\bar{\partial}$. Since $v\in {\rm Ker\,}\bar{\partial}$, it follows that $(w,v)=0$ if $w\in ({\rm Ker\,}\bar{\partial})^\bot$, so that the inequality (\ref{eq:BoundedFunctional}) holds for all $w\in {\rm Range\,}\bar{\partial}^\ast$.
     The Riesz representation theorem combined with the Hahn-Banach theorem then gives  some $u\in L^2_{(p,q-1)}({\mathbb C}^n)$ such that
     $$
     \|u\|_2 \le   2\sqrt{2} \|v\|_{\omega^{-\alpha'/\alpha},\alpha'}
     $$
      and
    $$
    (\bar{\partial}^\ast w,u)=(w,v),\ \ \ w\in {\rm Dom\,}\bar{\partial}^\ast,
    $$
    i.e. $\bar{\partial} u=v$. The first conclusion is then verified.

  For the second conclusion, we note that $u$ is an $L^2$ holomorphic function on ${\mathbb C}^n\backslash {\rm supp\,}v$, which contains a holomorphic cylinder $F({\mathbb C}\times {\mathbb D}^{n-1})$ where $F$ is a holomorphic injection from ${\mathbb C}\times {\mathbb D}^{n-1}$ into $\Omega$.
   Since
   $$
   \int_{{\mathbb C}\times {\mathbb D}^{n-1}} |u\circ F|^2 |{\rm det\,}F'|^2=\int_{F({\mathbb C}\times {\mathbb D}^{n-1})}|u|^2<\infty,
   $$
   it follows from Fubini's theorem that for almost all $z'\in {\mathbb D}^{n-1}$ the holomorphic function $h(\cdot,z')\in L^2({\mathbb C})$ where $h:=(u\circ F) {\rm det\,}F'$,  which has to vanish. Since $h$ is continuous, it follows that $h= 0$ on ${\mathbb C}\times {\mathbb D}^{n-1}$, i.e. $u= 0$ on $F({\mathbb C}\times {\mathbb D}^{n-1})$. By the uniqueness of holomorphic continuation, we conclude that $u= 0$ on the unbounded component of ${\mathbb C}^n\backslash {\rm supp\,}v$.
   \end{proof}

   Theorem \ref{th:L2Estimate} combined with Sobolev's inequality and Hardy's inequality gives

    \begin{corollary}\label{Cor:L2Estimate_3}
 If $n>1$, then for any $\bar{\partial}-$closed form $v\in L^2_{(p,q)}({\mathbb C}^n)\cap  L^{2n/(n+1)}_{(p,q)}({\mathbb C}^n)$,
   there exists a solution to $\bar{\partial} u =v$ which satisfies the estimate
 \begin{equation}\label{eq:L2Estimate_3}
 \|u\|_2\le C_n \left\|v \right\|_{2n/(n+1)}:=C_n \left[{\sum}'_{I,J}\int_{{\mathbb C}^n} |v_{I,J}|^{\frac{2n}{n+1}} \right]^{\frac{n+1}{2n}}.
 \end{equation}
 Moreover, if  $p=0$, $q=1$ and\/ ${\mathbb C}^n\backslash {\rm supp\,}v$ contains a holomorphic cylinder, then $u=0$ on the unbounded component of\/ ${\mathbb C}^n\backslash {\rm supp\,}v$.
\end{corollary}

    \begin{corollary}\label{Cor:L2Estimate_2}
 If $n>1$, then for any $\bar{\partial}-$closed $(p,q)$ form $v$ in ${\mathbb C}^n$  there exists a solution to $\bar{\partial} u =v$ which satisfies the estimate
 \begin{equation}\label{eq:L2Estimate_2}
 \|u\|_2\le \frac{2\sqrt{2}}{n-1}\left\|\max\{|z|,1\} v \right\|_2
 \end{equation}
 provided that the RHS of $(\ref{eq:L2Estimate_2})$ is finite. Moreover, if  $p=0$, $q=1$ and\/ ${\mathbb C}^n\backslash {\rm supp\,}v$ contains a holomorphic cylinder, then $u=0$ on the unbounded component of\/ ${\mathbb C}^n\backslash {\rm supp\,}v$.
\end{corollary}

 Corollary \ref{Cor:Hartogs} is a direct consequence of the following

\begin{theorem}\label{th:Hartogs}
Let $\Omega$ be a domain in ${\mathbb C}^n$, $n>1$, which satisfies $(\ref{eq:Hardy-Sobolev_4})$. Let $E$ be a closed set in $\Omega$ which satisfies
\begin{enumerate}
\item $\Omega\backslash E$ is connected;
\item there exists $r>0$ such that $
E_{r}:=\{z\in {\mathbb C}^n:d(z,E)\le r\}\subset \Omega;
$
\item ${\mathbb C}^n\backslash E_r$ contains a holomorphic cylinder.
\end{enumerate}
 If $f$ is holomorphic on $\Omega\backslash E$ such that $f\in L^2(E_r\backslash E_{r/2})$ and
 $$
\int_{E_r\backslash E_{r/2}}|f|^{\alpha'}\omega^{-\alpha'/\alpha} <\infty,
$$
then there is a holomorphic function $F$ on $\Omega$ such that $F|_{\Omega\backslash E}=f$.
\end{theorem}

\begin{proof}
 Let $\chi:{\mathbb R}\rightarrow [0,1]$ be a smooth function satisfying $\chi|_{(-\infty,1/2]}=0$ and $\chi|_{[1,\infty)}=1$. Set $v:=\bar{\partial}(\chi(d(\cdot,E)/r)f)$. Then $v$ is a $\bar{\partial}-$closed $(0,1)$ form on ${\mathbb C}^n$ which satisfies
 $$
 {\rm supp\,}v\subset E_r
 $$
 and since $|\nabla d(\cdot,E)|\le 1$ a.e., we have
   $$
 \|v\|_2^2 \le {\rm const}_r \int_{E_r\backslash E_{r/2}} |f|^2<\infty
 $$
  $$
 \|v\|_{\omega^{-\alpha'/\alpha},\alpha'}^2\le {\rm const}_r \left[\int_{E_r\backslash E_{r/2}} |f|^{\alpha'}\omega^{-\alpha'/\alpha}\right]^{2/\alpha'}<\infty.
 $$
  By Theorem \ref{th:L2Estimate}, there is a solution of $\bar{\partial} u=v$ such that $u= 0$ on the unbounded component of ${\mathbb C}^n\backslash {E}_r$, which contains a neighborhood of $\partial \Omega$. It is then easy to check that $F=\chi(d(\cdot,E)/r)f-u$ is the desired holomorphic extension of $f$.
\end{proof}

\section{ $\partial\bar{\partial}-$equation in ${\mathbb C}^n$ and applications}

Let $v=\sum_{j=1}^n v_{j\bar{k}} dz_j\wedge d\bar{z}_k$ be a $d-$closed smooth $(1,1)-$form on ${\mathbb C}^n$. Lelong \cite{Lelong} posed an elegant method of solving the  equation
\begin{equation}\label{eq:Lelong_2}
 \partial\bar{\partial}u=v
 \end{equation}
 by reducing it to the Poisson equation
\begin{equation}\label{eq:Lelong_1}
\frac{\Delta u}4=\sum_{j=1}^n \frac{\partial^2 u}{\partial z_j\partial \bar{z}_j}=\sum_{j=1}^n v_{j\bar{j}}=:{\rm Trace}(v).
 \end{equation}
 A key observation in \cite{Lelong} (see also \cite{MokSiuYau}) is that
  if $u$ is a solution of (\ref{eq:Lelong_1})  then $\psi:=|\partial\bar{\partial}u-v|^2$ is a subharmonic function on ${\mathbb C}^n$. It follows from the maximum principle that if $\psi$ vanishes at infinity then $u$ becomes a solution of (\ref{eq:Lelong_2}).

 \begin{lemma}\label{lm:Lelong}
 Let $v$ be a $d-$closed smooth $(1,1)-$form with compact support  in ${\mathbb C}^n$, $n>1$. Then there is a smooth solution of $(\ref{eq:Lelong_2})$ such that $u=0$ on the unbounded component of\/ ${\mathbb C}^n\backslash {\rm supp\,}v$.
 \end{lemma}

 \begin{proof}
 Without loss of generality, we assume that $v$ is a real $(1,1)-$form. We first solve the equation (\ref{eq:Lelong_1}). Let $C^\infty_0({\mathbb C}^n,{\mathbb R})$ denote the set of real-valued smooth functions with compact support in ${\mathbb C}^n$.  Sobolev's inequality in ${\mathbb C}^n={\mathbb R}^{2n}$ becomes
 \begin{equation}\label{eq:Lelong_Sobolev}
 \left[\int_{{\mathbb C}^n} |\phi|^{\frac{2n}{n-1}} \right]^{\frac{n-1}{2n}}\le C_n \left[\int_{{\mathbb C}^n} |\nabla \phi|^2\right]^{\frac12}
 \end{equation}
for all   $\phi\in C^\infty_0({\mathbb C}^n,{\mathbb R})$. It follows that $\|\nabla \cdot\|_2$ is a norm on $C^\infty_0({\mathbb C}^n,{\mathbb R})$. Let $H$ be the completion of $C^\infty_0({\mathbb C}^n,{\mathbb R})$ under this norm. Set $g=4{\rm Trace}(v)(\in C^\infty_0({\mathbb C}^n,{\mathbb R}))$. Since
\begin{eqnarray*}
\left|\int_{{\mathbb C}^n} g\cdot \phi \right| & \le &   \left[\int_{{\mathbb C}^n} |g|^{\frac{2n}{n+1}} \right]^{\frac{n+1}{2n}}\left[\int_{{\mathbb C}^n} |\phi|^{\frac{2n}{n-1}} \right]^{\frac{n-1}{2n}}\\
& \le & C_n \|g\|_{\frac{2n}{n+1}}\|\nabla \phi\|_2
\end{eqnarray*}
for all $\phi\in H$, it follows that $\phi\mapsto -\int_{{\mathbb C}^n} g\cdot \phi $ is a bounded linear functional on the Hilbert space $H$, so that there is, according to the Riesz Representation Theorem, an element $u\in H$ such that
\begin{equation}\label{eq:Lelong_Gradient}
\|\nabla u\|_2 \le C_n \|g\|_{\frac{2n}{n+1}}
\end{equation}
 and
\begin{equation}\label{eq:weak}
-\int_{{\mathbb C}^n} g\cdot \phi = \int_{{\mathbb C}^n} \nabla u\cdot \nabla\phi,
\end{equation}
i.e. $\Delta u=g$ holds in the weak sense. Since $g$ is smooth, so is $u$.

To show that $u$ satisfies (\ref{eq:Lelong_2}), it suffices to verify that $\partial\bar{\partial} u$ vanishes at infinity. Suppose ${\rm supp\,}v\subset B_R=\{|z|<R\}$. Since $u$ is harmonic on ${\mathbb C}^n\backslash B_R$, it follows from the mean-value property that  if $|z|>R$ then
\begin{eqnarray*}
 |u(z)| & = & C_n (|z|-R)^{-2n}\left|\int_{|\zeta-z|<|z|-R}u(\zeta)\right|\\
 & \le &C_n (|z|-R)^{1-n}\left|\int_{|\zeta-z|<|z|-R}|u(\zeta)|^{\frac{2n}{n-1}}\right|^{\frac{n-1}{2n}}\\
 & \le & C_n (|z|-R)^{1-n}\|g\|_{\frac{2n}{n+1}}
   \end{eqnarray*}
   where the first inequality follows from H\"older's inequality and the second follows from (\ref{eq:Lelong_Sobolev}) and  (\ref{eq:Lelong_Gradient}). Standard gradient estimates of harmonic functions (cf. \cite{GilbergTrudinger}, Theorem 2.10) imply that all second order derivatives  of $u$ at $z$ are bounded by $C_n |z|^{-n-1}\|g\|_{\frac{2n}{n+1}}$ provided $|z|\gg R$. Thus both $u$ and $\partial\bar{\partial}u$ have to vanish at infinity.

    Finally we fix a holomorphic cylinder ${\mathbb C}\times B'\subset {\mathbb C}^n\backslash B_R$ where $B'$ is a ball in ${\mathbb C}^{n-1}$. As $u$ is pluriharmonic on ${\mathbb C}\times B'$, we conclude that  for  all $z'\in B'$, $u(\cdot,z')$ is a harmonic function on ${\mathbb C}$ which vanishes at infinity. It follows from the maximum principle that  $u\equiv 0$ on ${\mathbb C}\times B'$. Since $u$ is pluriharmonic (hence real-analytic) on ${\mathbb C}^n\backslash {\rm supp\,}v$, it follows from the unique continuation property that $u\equiv 0$ on the unbounded component of ${\mathbb C}^n\backslash {\rm supp\,}v$.
    \end{proof}

    \begin{proof}[Proof of Theorem \ref{Thm:Hartogs_PH}]
 Let $\chi:{\mathbb R}\rightarrow [0,1]$ be a smooth function satisfying $\chi|_{(-\infty,1/2]}=0$ and $\chi|_{[1,\infty)}=1$. Set $v:=\partial\bar{\partial}(\chi(d(\cdot,E)/r)f)$ where $r<d(K,\partial \Omega)/2$. Then $v$ is a $d-$closed $(1,1)$ form with compact support in $E_r$.   By Lemma \ref{lm:Lelong}, there is a smooth solution of $\partial\bar{\partial} u=v$ such that $u= 0$ on the unbounded component of ${\mathbb C}^n\backslash {E}_r$, which contains a neighborhood of $\partial \Omega$. Note that $u$ may be chosen to be real-valued since it is obtained by solving the (real) equation
   $$
   \Delta u=\Delta(\chi(d(\cdot,E)/r)f).
   $$
    It is then easy to check that $F=\chi(d(\cdot,E)/r)f-u$ is the desired pluriharmonic extension of $f$ in view of the unique continuation property.
\end{proof}

\begin{remark}
One might formulate a result similar as Corollary \ref{Cor:Hartogs} when $E$ is not necessarily compact. We leave the details to the interested reader.
\end{remark}

\section{Hardy-Sobolev type inequalities}

\begin{proof}[Proof of Proposition \ref{prop:Integ_Ineq_mfd}]
 (1)  Recall first the classical Green's formula:
 $$
  \int_M u\Delta v d\mu=-\int_M \nabla u\cdot\nabla v d\mu,\ \ \ \forall\,u\in C^1(M),\,v\in C^\infty_0(M).
  $$
 By this formula, we obtain
 \begin{eqnarray*}
 && \int_M \frac{\phi^2}{\eta (-\psi)}\Delta\psi d\mu=-\int_M \nabla \psi \cdot \nabla \left[\frac{\phi^2}{\eta(-\psi)}\right]d\mu\\
  & = & -2 \int_M \phi \frac{\nabla \psi}{\eta(-\psi)}\cdot \nabla \phi d\mu- \int_M \phi^2 \frac{\eta'(-\psi)}{\eta^2(-\psi)}|\nabla \psi|^2 d\mu,
\end{eqnarray*}
so that
\begin{eqnarray}\label{eq:IntegByParts_1}
 && \int_M \frac{\phi^2}{\eta (-\psi)}\Delta\psi d\mu+\int_M \phi^2 \frac{\eta'(-\psi)}{\eta^2(-\psi)} |\nabla \psi|^2 d\mu\nonumber\\
  & =  &  -2 \int_M \phi \frac{\nabla \psi}{\eta(-\psi)}\cdot \nabla \phi d\mu\\
 & \le & \frac12  \int_M  \phi^2 \frac{\eta'(-\psi)}{\eta^2(-\psi)}|\nabla \psi|^2 d\mu+ 2 \int_M \frac{|\nabla\phi|^2}{\eta'(-\psi)}d\mu,\nonumber
\end{eqnarray}
from which (\ref{eq:Laplace-}) immediately follows.

(2) Similarly, we have
 \begin{eqnarray*}
 && \int_M \phi^2 \eta (\psi)\Delta\psi d\mu=-\int_M \nabla \psi \cdot \nabla \left(\eta(\psi)\phi^2\right)d\mu\\
  & = & -2 \int_M \phi \eta(\psi)\nabla \psi\cdot \nabla \phi d\mu- \int_M \phi^2 \eta'(\psi)|\nabla \psi|^2 d\mu,
\end{eqnarray*}
so that
\begin{eqnarray}\label{eq:IntegByParts_2}
 && \int_M \phi^2\eta (\psi)\Delta\psi d\mu+\int_M \phi^2 \eta'(\psi) |\nabla \psi|^2 d\mu\nonumber\\
  & = &  -2 \int_M \phi \eta(\psi)\nabla \psi\cdot \nabla \phi d\mu\\
 & \le & \frac12  \int_M  \phi^2 \eta'(\psi)|\nabla \psi|^2 d\mu+ 2 \int_M \frac{\eta^2(\psi)}{\eta'(\psi)}|\nabla\phi|^2 d\mu,\nonumber
\end{eqnarray}
which yields (\ref{eq:Laplace+}).
\end{proof}

A direct consequence of Proposition \ref{prop:Integ_Ineq_mfd} is

\begin{corollary}
Suppose furthermore that $\psi$ is subharmonic. Then
    \begin{equation}\label{eq:PoincareIneq_mfd}
\int_{M}\phi^2  \frac{\eta'(-\psi)}{\eta^2(-\psi)}{|\nabla\psi|^2} d\mu \le 4 \int_M \frac{|\nabla\phi|^2}{\eta'(-\psi)} d\mu
\end{equation}
    \begin{equation}\label{eq:PoincareIneq_mfd+}
   \int_ M \phi^2  \eta'(\psi){|\nabla\psi|^2} d\mu\le 4 \int_M \frac{\eta^2(\psi)}{\eta'(\psi)} |\nabla\phi|^2 d\mu
\end{equation}
for all $\phi\in C^\infty_0(M)$.
\end{corollary}

 The regularity assumptions on $\psi$ and $\phi$ can be relaxed significantly for domains in ${\mathbb R}^n$. As usual, we denote by $W^{k,2}(\Omega)$ the Sobolev space of all functions whose derivatives of order $\le k$ are $L^2$, and $W^{k,2}_0(\Omega)$ the completion of $C^\infty_0(\Omega)$ in $W^{k,2}(\Omega)$.

 \begin{proposition}\label{prop:PoincareIneq}
Let  $\psi< -\gamma<0$ be a subharmonic function on a domain $\Omega\subset {\mathbb R}^n$. Let $\eta:(0,\infty)\rightarrow (0,\infty)$ be a $C^1$ function such that $\eta'$ is a positive continuous decreasing function with $1/\eta'(-\psi)\in L^1_{\rm loc}(\Omega)$. Then $\kappa_\eta(-\psi)\in W^{1,2}_{\rm loc}(\Omega)$ where $\kappa_\eta(t):=\int_\gamma^t \frac{\sqrt{\eta'(s)}}{\eta(s)}ds$, $t>\gamma$. Moreover, one has
\begin{equation}\label{eq:PoincareIneq-}
\int_{\Omega}\phi^2  \frac{\eta'(-\psi)}{\eta^2(-\psi)}{|\nabla\psi|^2}  \le 4 \int_\Omega \frac{|\nabla\phi|^2}{\eta'(-\psi)},\ \ \ \phi\in W^{1,2}_0(\Omega).
\end{equation}
\end{proposition}

\begin{proof}
   Fix arbitrary open sets $\Omega'\subset\subset \Omega''\subset\subset \Omega$.   We take a sequence of smooth  subharmonic functions $\{\psi_j\}$ in a neighborhood of $\overline{\Omega''}$  such that $\psi_j<-\gamma$ and $\psi_j\downarrow \psi$  as $j\rightarrow \infty$.
  By Schwarz's inequality, we have
 $$
 \kappa^2_\eta(t)\le (t-\gamma)\int_\gamma^t \frac{\eta'}{\eta^2}=(t-\gamma)(\eta(\gamma)^{-1}-\eta(t)^{-1})\le t/\eta(\gamma),
 $$
 so that
 $$
 \kappa_\eta(-\psi_j)\le \sqrt{-\psi_j/\eta(\gamma)}\le  \sqrt{-\psi/\eta(\gamma)} \in L^2(\Omega').
 $$
 Choose $\phi\in C_0^\infty(\Omega'')$ with $\phi|_{\Omega'}=1$. It follows from (\ref{eq:PoincareIneq_mfd}) that
$$
\int_{\Omega'} \left|\nabla \kappa_\eta(-\psi_j)\right|^2  \le  4 \int_{\Omega''}\frac{|\nabla\phi|^2}{\eta'(-\psi_j)}
 \le  4 \int_\Omega\frac{|\nabla\phi|^2}{\eta'(-\psi)}.
$$
It follows that $\{\kappa_\eta(-\psi_j)\}$ is uniformly bounded in $W^{1,2}(\Omega')$, so that there is a subsequence $\{\kappa_\eta(-\psi_{j_k})\}$ converging weakly to an element $\varphi$ in $W^{1,2}(\Omega')$. As $ \kappa_\eta(-\psi_j) \rightarrow \kappa_\eta(-\psi)$ in $L^2(\Omega')$ in view of the dominated convergence theorem, we conclude that $\varphi=\kappa_\eta(-\psi)$ a.e. on $\Omega'$, so that $\kappa_\eta(-\psi)\in W^{1,2}(\Omega')$ and (\ref{eq:PoincareIneq-}) holds for all $\phi\in C_0^\infty(\Omega)$. For any $\phi\in W^{1,2}_0(\Omega)$, we first note that if $\psi$ is smooth then $(\ref{eq:PoincareIneq-})$ holds for $\phi$. In general, for any $1\le \nu\le n$ we have
\begin{eqnarray*}
\int_{\Omega} \left|\frac{\partial \kappa_\eta(-\psi)}{\partial x_\nu}\right|^2 \phi^2 & = &\sup_{g\in C^\infty_0(\Omega),\|g\|_{L^2(\Omega)}=1}  \left|\int_{\Omega}\frac{\partial \kappa_\eta(-\psi)}{\partial x_\nu} \phi g\right|^2\\
& = &   \sup_{g\in C^\infty_0(\Omega),\|g\|_{L^2(\Omega)}=1} \left|\lim_{k\rightarrow \infty}\int_{\Omega}\frac{\partial \kappa_\eta(-\psi_{j_k})}{\partial x_\nu} \phi g\right|^2 \\
                                             & \le &  \limsup_{k \rightarrow \infty}\int_{\Omega} \left|\frac{\partial \kappa_\eta(-\psi_{j_k})}{\partial x_\nu}\right|^2 \phi^2.
                                            \end{eqnarray*}
  It follows that
  \begin{eqnarray*}
   \int_\Omega |\nabla \kappa_\eta(-\psi)|^2 \phi^2 & \le & \limsup_{k \rightarrow \infty} \int_\Omega |\nabla \kappa_\eta(-\psi_{j_k})|^2\phi^2\\
   & \le & 4 \limsup_{k \rightarrow \infty} \int_\Omega \frac{|\nabla \phi|^2}{\eta'(-\psi_{j_k})} \\
   & = & 4 \int_\Omega \frac{|\nabla \phi|^2}{\eta'(-\psi)}
  \end{eqnarray*}
  in view of the monotonic convergence theorem.
\end{proof}

\begin{example}
 Let $\eta(t)= t^\alpha$, $t\in (0,\infty)$, where $0<\alpha\le 1$. We have $\eta'(t)=\alpha t^{\alpha-1}$ and $1/\eta'(-\psi)=\frac1\alpha (-\psi)^{1-\alpha}\in L^1_{\rm loc}(\Omega)$ if $\psi<0$, so that
 \begin{equation}\label{eq:PoincareIneq_3}
\int_{\Omega}\phi^2\frac{|\nabla\psi|^2}{|\psi|^{1+\alpha}} \le \frac4{\alpha^2} \int_\Omega |\psi|^{1-\alpha} |\nabla\phi|^2.
\end{equation}
The case $\alpha=1$, i.e.
\begin{equation}\label{eq:PoincareIneq_Log}
\int_\Omega \phi^2 |\nabla \psi|^2/\psi^2\le 4\int_\Omega |\nabla \phi|^2,
\end{equation}
 is  particularly useful. For instance, we get Hardy's inequality $(\ref{eq:Hardy-Sobolev_2})$ by letting $\psi=-|x|^{2-n}$ in $(\ref{eq:PoincareIneq_Log})$.
\end{example}

We also need the following result.

\begin{proposition}\label{prop:Laplace}
 If $\psi$ is a negative subharmonic function on a domain $\Omega\subset {\mathbb R}^n$, then
    $$
 \int_\Omega \phi^2 \Delta \psi\le 3\pi \int_\Omega (1+\psi^2)|\nabla \phi|^2,\ \ \ \phi\in C^\infty_0(\Omega).
 $$
\end{proposition}

\begin{proof}
We take a decreasing sequence of smooth subharmonic functions $\psi_j<0$ defined in a neighborhood of ${\rm supp\,}\phi$ such that $\psi_j\downarrow \psi$.  Applying (\ref{eq:Laplace-}) with $\eta(t)=\pi+\arctan t$, we have
  $$
  \int_\Omega \frac{\phi^2}{\eta(-\psi_j)}\Delta\psi_j \le 2 \int_\Omega (1+\psi^2_j)|\nabla \phi|^2\le 2 \int_\Omega (1+\psi^2)|\nabla \phi|^2.
  $$
  Since $\eta\le 3\pi/2$, it follows that
  \begin{eqnarray*}
   \int_\Omega \phi^2 \Delta\psi & = & \lim_{j\rightarrow \infty}  \int_\Omega \phi^2 \Delta\psi_j\\
   & \le & 3\pi  \int_\Omega (1+\psi^2)|\nabla \phi|^2.
  \end{eqnarray*}
\end{proof}

\section{Subharmonic functions and BMO}

 Let $cB$ denote the ball which has the same center as $B$ but whose radius is expanded by the factor $c$. Then we have

\begin{lemma}\label{lm:Doubling}
 Let $\psi$ be a negative subharmonic function on ${\mathbb R}^n$ with $n>2$. Then for any $0<\gamma\le 1$ one has
 \begin{equation}\label{eq:Doubling}
 \int_{2B} |\psi|^\gamma \le 2^n \int_{B} |\psi|^\gamma
 \end{equation}
 for all balls $B\subset {\mathbb R}^n$, i.e. $|\psi|^\gamma$ is a doubling measure.
\end{lemma}

\begin{proof}
Let $\sigma_n$ be the volume of the unit sphere in ${\mathbb R}^n$. Since $\varphi:=-(-\psi)^\gamma$ is a subharmonic function on ${\mathbb R}^n$, it follows that the mean value
$$
M_\varphi(x,r)=\int_{|y|=1} \varphi(x+ry)d\sigma(y)/\sigma_n
$$
is an increasing function of $r$ for any $x\in {\mathbb R}^n$ (see \cite{HormanderConvexity}, Theorem 3.2.2), i.e.
$$
M_{|\psi|^\gamma} (x,r)=\int_{|y|=1} |\psi|^\gamma (x+ry)d\sigma(y)/\sigma_n
$$
is a decreasing function of $r$. Given a ball $B=B(x,r)$, we have
\begin{eqnarray*}
  \int_{2B\backslash B} |\psi|^\gamma & = & \int_r^{2r} M_{|\psi|^\gamma} (x,t) t^{n-1}\sigma_n dt \\
  & \le & (2^n-1) \frac{ \sigma_n}{n} r^{n} M_{|\psi|^\gamma} (x,r)\\
  & \le & (2^n-1) \int_{0}^{r} M_{|\psi|^\gamma} (x,t) t^{n-1}\sigma_n dt\\
  & = & (2^n-1) \int_{B} |\psi|^\gamma,
\end{eqnarray*}
from which (\ref{eq:Doubling}) immediately follows.
\end{proof}

 \begin{proof}[Proof of Theorem \ref{th:BMO_Subharmonic}/(2)]
 By Lemma \ref{eq:Doubling}, it suffices to verify the following\/ {\it weak\/} reverse H\"older inequality:
   \begin{equation}\label{eq:WeakHolder}
 \left[\frac1{|B|}\int_B |\psi|^{\frac{\gamma n}{n-2}}\right]^{\frac{n-2}n}\le \frac{C_{n}}{(1-\gamma)^2}\cdot \frac1{|2B|}\int_{2B} |\psi|^\gamma.
 \end{equation}
    Choose $\chi\in C_0^\infty(2B)$ such that $\chi|_{\frac32 B}=1$ and $|\nabla \chi|\le 3/r$.
 Set $\varphi=(-\psi)^{\gamma/2}$. Applying (\ref{eq:PoincareIneq_3}) with $\alpha=1-\gamma$ and $\phi=\chi$, we obtain
 $$
 \int_{{\mathbb R}^n} \chi^2 |\nabla \varphi|^2 \le \left[\frac{\gamma}{1-\gamma}\right]^2
 \int_{{\mathbb R}^n} \varphi^2 |\nabla \chi|^2,
 $$
 so that
 $$
 \int_{\frac32 B}  |\nabla \varphi|^2 \le \left[\frac{\gamma}{1-\gamma}\right]^2  9 r^{-2}
 \int_{2B} \varphi^2.
 $$
  Let $0\le \kappa\le 1$ be a smooth function supported in $\frac32 B$ such that $\kappa|_{B}=1$ and
  $|\nabla \kappa|\le 3/r$.
Sobolev's inequality implies that
 \begin{eqnarray*}
 \left[\int_{{\mathbb R}^n} |\kappa \varphi|^{\frac{2n}{n-2}} \right]^{\frac{n-2}n}
 & \le & C_n \int_{{\mathbb R}^n} |\nabla (\kappa \varphi)|^2\\
 & \le & 2C_n\left[\int_{{\mathbb R}^n} \varphi^2 |\nabla \kappa|^2 + \int_{{\mathbb R}^n}\kappa^2 |\nabla \varphi|^2 \right]\\
 & \le & {\rm const}_{n}\,(1-\gamma)^{-2} r^{-2}
 \int_{2B} \varphi^2
  \end{eqnarray*}
  from which (\ref{eq:WeakHolder}) immediately follows.
  \end{proof}

  \begin{proof}[Proof of Theorem \ref{th:BMO_Subharmonic}/(1)]
 Recall that the  capacity of a compact set $K\subset {\mathbb R}^n$ is defined by
   $$
   {\rm Cap}(K)=\inf \int_{{\mathbb R}^n} |\nabla \phi|^2
   $$
   where the infimum is taken over all ${\phi\in C^\infty_0({\mathbb R}^n)}$ such $\phi|_K=1$.
   Let  $B$ be a ball with radius $r$. Then we have
 $$
 {\rm Cap}(B)=(n-2)\sigma_n r^{n-2}
 $$
  (see e.g. \cite{Grigoryan}, p.\,17).
 We take a sequence of functions $\phi_j\in C^\infty_0({\mathbb R}^n)$ such that $\phi_j|_{B}=1$ and
 $$
 \int_{{\mathbb R}^n}|\nabla \phi_j|^2 \rightarrow {\rm Cap}(B)\  \  \  {\rm as\  \  \  } j\rightarrow \infty.
 $$
  Set $f=\log(-\psi)$. By (\ref{eq:PoincareIneq_Log}), we have
$$
\int_B |\nabla f|^2\le \int_\Omega \phi^2_j |\nabla \psi|^2/\psi^2\le 4\int_\Omega |\nabla \phi_j|^2\rightarrow 4{\rm Cap}(B).
 $$
 Let $\mu_n$ be the first Neumann eigenvalue of the unit ball $B_1$ in ${\mathbb R}^n$, i.e.
 $$
 \mu_n=\inf\frac{\int_{B_1}|\nabla \phi|^2}{\int_{B_1}|\phi|^2}
 $$
where the infimum is taken over all $\phi\in W^{1,2}(B_1)$ with $\int_{B_1} \phi=0$.  It follows that
  $$
 \int_B |f-f_B|^2  \le  \frac{r^2}{\mu_n}\int_B |\nabla f|^2
 \le  \frac{4r^2}{\mu_n} {\rm Cap}(B).
$$
Let $\omega_n=\sigma_n/n$ be the volume of the unit ball in ${\mathbb R}^n$. It follows that
 $$
\frac1{ |B|}\int_B |f-f_B|^2\le \frac{4r^2}{\mu_n}\frac{{\rm Cap}(B)}{\omega_n r^n}=\frac{4}{\mu_n}n(n-2).
 $$
  By the Schwarz inequality, we have
 \begin{equation}\label{eq:BMO}
 \frac1{|B|}\int_B |f-f_B|\le \left[\frac1{|B|}\int_B |f-f_B|^2\right]^{1/2}\le 2\sqrt{n(n-2)/\mu_n}.
 \end{equation}
 By the following proposition, we get
 $$
 \|f\|_{\rm BMO} \le 2 \sqrt{\frac{(n-2)n(n+6)}{(n+2)(n+4)}}<2\sqrt{n-2}.
 $$
  \end{proof}

 \begin{proposition}\label{prop:NeumannEigenvalue}
 \begin{equation}\label{eq:NeumannEigenvalue}
 \frac{(n+2)(n+4)}{n+6}< \mu_n < {n+2}.
 \end{equation}
 \end{proposition}

 \begin{proof}
 We review some basic properties of Bessel functions, following the classical book of Watson \cite{Watson}. Recall that the Bessel function $J_\nu$ of order $\nu$ is given by
 \begin{equation}\label{eq:Bessel_1}
 J_\nu(x)=(x/2)^{\nu}\sum_{k=0}^\infty \frac{(-1)^k}{k!}(x/2)^{2k}\frac1{\Gamma(k+\nu+1)},
 \end{equation}
 which satisfies the Bessel equation
 \begin{equation}\label{eq:Bessel_2}
 \frac{d^2 y}{dx^2}+\frac1x \frac{dy}{dx}+\left[1-\frac{\nu^2}{x^2}\right]y=0.
 \end{equation}
 Moreover, the function $J_\nu$ satisfies the following function equations
 \begin{equation}\label{eq:Bessel_3}
 J'_{\nu+2}(x)=\frac{2(\nu+1)}{x}\left[1-\frac{\nu(\nu+2)}{x^2}\right]J_\nu(x)-\left[1-\frac{2(\nu+1)(\nu+2)}{x^2}\right]J'_\nu(x)
 \end{equation}
 \begin{equation}\label{eq:Bessel_4}
 J_{\nu+2}(x)=-\left[1-\frac{2\nu(\nu+1)}{x^2}\right]J_\nu(x)-\frac{2(\nu+1)}{x}J'_\nu(x),
 \end{equation}
 and if $j_\nu$ and $j_\nu'$ are the lowest positive root of $J_\nu$ and $J'_\nu$ respectively, then
 \begin{equation}\label{eq:Bessel_Root_1}
\sqrt{\nu(\nu+2)}<j_\nu < \sqrt{2(\nu+1)(\nu+3)}
\end{equation}
\begin{equation}\label{eq:Bessel_Root_2}
\sqrt{\nu(\nu+2)} < j_\nu' < \sqrt{2\nu(\nu+1)}
\end{equation}
\begin{equation}\label{eq:Bessel_Root_3}
j_\nu'< j_\nu < j_{\nu+1} < j_{\nu+2}
\end{equation}
(cf. \cite{Watson}, p. 485--486).
It is known from \cite{Weinberger} that $\sqrt{\mu_n}$ is exactly the lowest positive root of $\varphi'(x)$ where $\varphi$ satisfies the following Bessel-type equation
\begin{equation}\label{eq:Bessel_5}
\frac{d^2 y}{dx^2}+\frac{n-1}{x}\frac{dy}{dx}+\left[1-\frac{n-1}{x^2}\right] y=0.
\end{equation}
Equivalently, $\sqrt{\mu_n}$ is the lowest positive root of the equation
\begin{equation}\label{eq:Bessel_6}
x J'_{n/2}(x)-(n/2-1)J_{n/2}(x)=0
\end{equation}
(cf. \cite{Payne}, p.\,288). In fact, a straightforward calculation shows that the function
$$
\varphi(x)=x^{1-n/2}J_{n/2}(x)
$$
 satisfies the equation (\ref{eq:Bessel_5}).
By (\ref{eq:Bessel_4}), we have
\begin{equation}\label{eq:Bessel_9}
 J_{\frac{n}2+2}(x)=-\left[1-\frac{n(\frac{n}2+1)}{x^2}\right]J_{n/2}(x)-\frac{n+2}{x}J'_{n/2}(x).
 \end{equation}
Take $x=\sqrt{\mu_n}$ and substitute (\ref{eq:Bessel_6}) into (\ref{eq:Bessel_9}), we have
\begin{equation}\label{eq:Bessel_10}
 J_{\frac{n}2+2}(\sqrt{\mu_n})=J_{n/2}(\sqrt{\mu_n})\left[\frac{n+2}{\mu_n}-1\right].
 \end{equation}
 It is known that $\mu_n$ is always less than the first Dirichlet eigenvalue of the unit ball in ${\mathbb R}^n$, which equals to $j_{n/2-1}^2$ (see e.g. \cite{PayneSurvey}), so that
 $$
 \sqrt{\mu_n}<j_{n/2} < j_{\frac{n}2+2}.
 $$
Thus (\ref{eq:Bessel_10}) implies
 $$
 \mu_n<{n+2}.
 $$
By (\ref{eq:Bessel_3}), we have
\begin{equation}\label{eq:Bessel_7}
 J'_{\frac{n}2+2}(x)=\frac{n+2}{x}\left[1-\frac{\frac{n}2(\frac{n}2+2)}{x^2}\right]J_{n/2}(x)-\left[1-\frac{(n+2)({\frac{n}2+2})}{x^2}\right]J'_{n/2}(x).
 \end{equation}
 Take $x=\sqrt{\mu_n}$ and substitute (\ref{eq:Bessel_6}) into (\ref{eq:Bessel_7}), we get
 \begin{equation}\label{eq:Bessel_8}
 J'_{\frac{n}2+2}(\sqrt{\mu_n})=\frac{J_{n/2}(\sqrt{\mu_n})}{\sqrt{\mu_n}}\left[\frac{n}2+3-\frac{(n+2)({\frac{n}2+2})}{\mu_n} \right].
 \end{equation}
Suppose one has
$$
\mu_n\le {\frac{(n+2)(\frac{n}2+2)}{\frac{n}2+3}},
$$
then the RHS of (\ref{eq:Bessel_8}) is nonpositive. Since
$$
j'_{\frac{n}2+2} > \sqrt{\left[\frac{n}2+2\right]\left[\frac{n}2+4\right]}> \sqrt{n+2}>\sqrt{\mu_n},
$$
it follows that the LHS of (\ref{eq:Bessel_8}) is positive, which is absurd. Thus
$$
\mu_n > {\frac{(n+2)(n+4)}{n+6}}.
$$
 \end{proof}

 It is also not known whether the left side of (\ref{eq:BMO_Bound}) can be bounded below by $c\sqrt{n}$ for some absolute constant $c>0$. If it were true, then
 $
 \|\log |x|\|_{\rm BMO} \ge c /\sqrt{n}.
 $
 Unfortunately, we can only show a worse bound
 $$
  \|\log |x|\|_{\rm BMO}\ge \frac{2}{en}.
   $$
   To see this, it suffices to  estimate the mean oscillation of $f:=\log 1/|x|$ over the unit ball $B_1$. A straightforward calculation gives
   $$
   f_{B_1}=\frac{\sigma_n}{|B_1|}\int_0^1 (\log 1/r)r^{n-1}dr=1/n
   $$
   and
   \begin{eqnarray*}
   \int_{B_1}|f-f_{B_1}| & = & \sigma_n \int_0^1 |\log 1/r -1/n| r^{n-1}dr\\
   & = & \sigma_n \int_0^\infty |s-1/n|e^{-ns} ds\\
   & = & \frac{2\sigma_n}{e n^2}.
   \end{eqnarray*}
   It follows immediately that
   $$
   \|f\|_{\rm BMO}\ge \frac1{|B_1|} \int_{B_1} |f-f_{B_1}|\ge \frac2{en}.
   $$

   \begin{problem}
   What is the actual behavior of\/ $\|\log |x|\|_{\rm BMO}$ as $n\rightarrow \infty$?
      \end{problem}

      \section{Plurisubharmonic functions and BMO}

    The proof of Theorem \ref{th:BMO_PSH} is based on a number of lemmata. Let $SH^-$ (resp. $PSH^-$) denote the set of negative subharmonic (resp. psh) functions.

  \begin{lemma}\label{lm:BMO_PSH_1}
  Let $B_R=\{z\in {\mathbb C}:|z|<R\}$ with $R\le 1/2$. If $\psi\in SH^-({B_{4R}})$, then
  \begin{equation}\label{eq:BMO_PSH_1}
  \frac1{|B|}\int_B |\psi-\psi_B|\le C_0 R^{-2} {|\log R|} \int_{B_{3R}}(1+\psi^2)
  \end{equation}
  for all balls $B\subset B_{R}$. Here $C_0>0$ is an absolute constant.
  \end{lemma}

  \begin{proof}
 Applying Proposition \ref{prop:Laplace} with  $\phi\in C^\infty_0({B_{3R}})$ such that $\phi|_{B_{2R}}=1$ and $|\nabla \phi|\le C_0/R$, we  conclude that
    \begin{equation}\label{eq:BMO_PSH_2}
   \int_{B_{2R}} \Delta\psi\le {C_0}{R^{-2}}  \int_{B_{3R}}(1+\psi^2).
  \end{equation}
  Recall that the (negative) Green function of $B_{2R}$ is given by
  $$
  g_R(z,w)=\log |z-w|+\log \frac{2R}{|4R^2-z\bar{w}|}.
  $$
  For any $z\in B_{2R}$, the Riesz decomposition theorem (cf. \cite{HormanderConvexity}, Theorem 3.3.6) gives
  \begin{eqnarray*}
   \psi(z) & = & \frac1{2\pi}\int_{B_{2R}} g_R(z,\zeta)\Delta \psi(\zeta) +h(z)\\
   &  = & \frac1{2\pi} \int_{B_{2R}} \log |z-\zeta| \Delta \psi(\zeta)+\frac1{2\pi}\int_{B_{2R}} \log \frac{2R}{|4R^2-z\bar{\zeta}|}\Delta \psi(\zeta) + h(z)\\
   & = :& u(z)+v(z)+h(z)
  \end{eqnarray*}
  where $h$ is the\/ {\it smallest}\/ harmonic majorant of $\psi$, which naturally satisfies
  $$
  \psi\le h\le 0.
  $$
  It is easy to see that
  \begin{equation}\label{eq:BMO_PSH_3}
   \|v\|_{L^\infty(B_R)}\le C_0 {|\log R|} \int_{B_{2R}} \Delta\psi\le C_0 R^{-2} {|\log R|}  \int_{B_{3R}}(1+\psi^2).
  \end{equation}
  Since $h\le 0$ is harmonic on $B_{2R}$, it follows from the mean-value property that for $z\in B_R$
  \begin{eqnarray}\label{eq:BMO_PSH_4}
  -h(z) & = &¡¡\frac{1}{|B(z,R)|}\int_{B(z,R)} (-h)\nonumber\\
   & \le & C_0 R^{-2} \int_{B_{2R}} (-\psi)\nonumber\\
   & \le & C_0 R^{-2} \int_{B_{3R}}(1+\psi^2).
  \end{eqnarray}
  For any ball $B\subset B_R$ and $z\in B$, we have
  \begin{eqnarray*}
  2\pi [u(z)-u_B] & = & \int_{B_{2R}}\log |z-\zeta|\, \Delta\psi(\zeta)-\frac1{|B|}\int_{w\in B} \int_{\zeta\in B_{2R}} \log |w-\zeta|\,\Delta\psi(\zeta)\\
   & = & \int_{B_{2R}}[\log |z-\zeta|-(\log |\cdot-\zeta|)_B] \Delta\psi(\zeta),
  \end{eqnarray*}
  so that
  \begin{eqnarray}\label{eq:BMO_PSH_5}
   \frac1{|B|}\int_B |u-u_B| & \le & \frac1{2\pi}\|\log |z|\|_{\rm BMO}  \int_{B_{2R}} \Delta\psi\nonumber\\
   & \le & {C_0}{R^{-2}}  \int_{B_{3R}}(1+\psi^2).
  \end{eqnarray}
  Clearly, (\ref{eq:BMO_PSH_3})-(\ref{eq:BMO_PSH_5}) imply (\ref{eq:BMO_PSH_1}).
  \end{proof}

  Given $a\in {\mathbb C}^n$ we define the polydisc
  $$
  P(a,R)=\{z\in {\mathbb C}^n: \max_j |z_j-a_j|<R\}.
  $$

  \begin{lemma}\label{lm:BMO_PSH_2}
   For any $\alpha\ge 1$, there exists a constant $C_{n,\alpha}$ depending only on $n,\alpha$ such that
      $$
   \int_{|z_1|<R}\int_{z'\in P(0',r)}|\psi(z_1,z')|^\alpha \le C_{n,\alpha} r^{2n-2} \int_{|z_1|<R} |\psi(z_1,0')|^\alpha
   $$
   for all $\psi\in PSH^-(\overline{P(0,4R)})$ and $r<R$.
  \end{lemma}

  \begin{proof}
 We first recall a consequence of the Riesz decomposition theorem that if $u$  is a negative subharmonic function in a neighborhood of the unit closed disc in ${\mathbb C}$, then
  \begin{equation}\label{eq:Hormander}
  \int_{|z|<1/2} |u|^\alpha \le C_\alpha |u(0)|^\alpha
  \end{equation}
  where $C_\alpha>0$ is a constant depending only on $\alpha$ (cf. \cite{HormanderConvexity}, p.\,230). In the case of $n$ complex variables we consider a negative psh function $u$ in $P(0,2)$. Then we have
  \begin{eqnarray*}
   \int_{P(0,1/2)} |u|^p & \le & C_\alpha \int_{|z_1|<1/2}\cdots\int_{|z_{n-1}|<1/2} |u(z_1,\cdots,z_{n-1},0)|^\alpha\\
   & \le & \cdots \le C_\alpha^{n-1} \int_{|z_1|<1/2} |u(z_1,0')|^\alpha.
  \end{eqnarray*}
  We conclude the proof by letting $u(z)=\psi(2R z_1,2r z')$.
  \end{proof}

  \begin{lemma}\label{lm:BMO_PSH_3}
   If $\psi\in SH^-(\overline{P(0,6R)})$ with $R\le 1/3$, then
  \begin{equation}\label{eq:BMO_PSH_6}
  \frac1{|P|}\int_P |\psi-\psi_P|\le C_n  R^{-2} {|\log R|} \sum_{k=1}^n \int_{|z_k|<3R}[1+\psi^2(0,\cdots,0,z_k,0,\cdots,0)]
  \end{equation}
  for any polydisc $P=P(0,r)$  with $r<R$.
  \end{lemma}

  \begin{proof}
  We write
  $
  P=\prod_{j=1}^n B^j
  $
  where $B^j=\{z_j:|z_j|<r\}$. Then for $z\in P$ we have
  \begin{eqnarray*}
  \psi(z)-\psi_P & = & \psi(z_1,z_2,\cdots,z_n)-\psi(\cdot,z_2,\cdots,z_n)_{B^1}+\cdots\\
  && + \psi(\cdots,z_{k},\cdots,z_n)_{B^1\cdots B^{k-1}}- \psi(\cdots,z_{k+1},\cdots,z_n)_{B^1\cdots B^{k}}+\cdots\\
  && +  \psi(\cdots,z_n)_{B^1\cdots B^{n-1}}- \psi_{B^1\cdots B^{n}}
  \end{eqnarray*}
  where
  \begin{eqnarray*}
 &&  \psi(\cdots,z_{k},\cdots,z_n)_{B^1\cdots B^{k-1}}\\
  & = & \frac1{|B^1|\cdots|B^{k-1}|}\int_{\zeta_1\in B^1}\cdots \int_{\zeta_{k-1}\in B^{k-1}}\psi(\zeta_1,\cdots,\zeta_{k-1},z_k\cdots,z_n).
  \end{eqnarray*}
 Since
  \begin{eqnarray*}
  && |\psi(\cdots,z_{k},\cdots,z_n)_{B^1\cdots B^{k-1}}- \psi(\cdots,z_{k+1},\cdots,z_n)_{B^1\cdots B^{k}}|\\
  & \le & |\psi(\cdots,z_{k},\cdots,z_n)- \psi(\cdots,z_{k+1},\cdots,z_n)_{B^{k}}|_{B^1\cdots B^{k-1}},
  \end{eqnarray*}
  it follows that
   \begin{eqnarray*}
  && \frac1{|B^k|} \int_{z_k\in B^k} |\psi(\cdots,z_{k},\cdots,z_n)_{B^1\cdots B^{k-1}}- \psi(\cdots,z_{k+1},\cdots,z_n)_{B^1\cdots B^{k}}|\\
  & \le & \left| \frac1{|B^k|} \int_{z_k\in B^k} \left|\psi(\cdots,z_{k},\cdots,z_n)- \psi(\cdots,z_{k+1},\cdots,z_n)_{B^{k}}\right| \right|_{B^1\cdots B^{k-1}}\\
  & \le &  C_0 R^{-2} {|\log R|} \left[\int_{|z_k|<3R} (1+\psi^2(\cdots,z_{k},\cdots,z_n)) \right]_{B^1\cdots B^{k-1}}
    \end{eqnarray*}
    in view of Lemma \ref{lm:BMO_PSH_1}.
    Thus
    \begin{eqnarray*}
     \frac1{|P|}\int_P |\psi-\psi_P|
     & \le &  C_0 R^{-2} {|\log R|}\sum_{k=1}^n \frac1{|B^1|\cdots|B^{k-1}|}\\
     && \cdot \int_{B^1}\cdots\int_{B^{k-1}}\left[\int_{|z_k|<3R} (1+\psi^2(\cdots,z_{k},\cdots)) \right]_{B^1\cdots B^{k-1}B^{k+1}\cdots B^n}\\
     & \le &  C_n  R^{-2} {|\log R|} \sum_{k=1}^n \int_{|z_k|<3R}[1+\psi^2(0,\cdots,0,z_k,0,\cdots,0)]
    \end{eqnarray*}
  in view of Lemma \ref{lm:BMO_PSH_2}.
  \end{proof}

  We also need the following elementary fact:

  \begin{lemma}\label{lm:BMO_PSH_4}
 Let $V$ be a bounded domain in ${\mathbb C}^n$ and $f\in L^1(V)$. For any domain $W\subset V$ one has
 $$
 \frac1{|W|}\int_W |f-f_W|\le  \frac{2|V|}{|W|}\left[\frac1{|V|}\int_V |f-f_V|\right].
 $$
   \end{lemma}

  \begin{proof}
  Since
  \begin{eqnarray*}
  |f_W-f_V | & \le &  \frac1{|W|}\int_{W} |f-f_{V}|\\
  & \le &  \frac1{|W|}\int_{V} |f-f_V|\\
  & \le &  \frac{|V|}{|W|}\left[\frac1{|V|}\int_V |f-f_V|\right],
  \end{eqnarray*}
  it follows that
  \begin{eqnarray*}
   \frac1{|W|}\int_{W} |f-f_W| & \le &   \frac1{|W|}\int_{V} |f-f_V|+|f_W-f_V|\\
   & \le & \frac{2|V|}{|W|}\left[\frac1{|V|}\int_V |f-f_V|\right].
  \end{eqnarray*}
  \end{proof}

  \begin{lemma}\label{lm:BMO_PSH_5}
  For any $\psi\in PSH^-(B(0,R))$, there exists a positive number $m$ depending only on $n$ and $(-\psi)_{B(0,R/5)}$ such that there are complex lines $L_1,\cdots,L_n$ which are orthogonal each other and
  $
  L_j\cap S_m\neq \emptyset
  $
  for all $j$, where
   $$
  S_m=\{z\in B(0,R/5):\psi(z)>-m\}.
  $$
  \end{lemma}

  \begin{proof}
  Set $S_m^c=B(0,R/5)-S_m$. By Chebychev's inequality, we have
  $$
  \int_{B(0,R/5)} |\psi|\ge m |S_m^c|,
  $$
  so that
  \begin{eqnarray}\label{eq:VolumeEstimate}
  |S_m|  \ge  |B(0,R/5)|-\frac1m \int_{B(0,R/5)} |\psi|
   >  \frac12|B(0,R/5)|
  \end{eqnarray}
  provided
  $$
  m>2(-\psi)_{B(0,R/5)}.
  $$
  We  choose a number $0<c_n< 1/5$ such that
  $$
  \left|\{z\in B(0,R/5):|z_1|<c_nR\}\right|+|\{z\in B(0,R/5):|z_2|<c_nR\}|<\frac14|B(0,R/5)|.
  $$
  Set
  $$
  S_m'=S_m\cap \{z:\min\{|z_1|,|z_2|\}>c_n R\}.
  $$
  Then we have
  $$
  |S_m'|\ge |S_m|-\frac14|B(0,R/5)|>\frac12 |S_m|.
  $$
  We define a smooth homeomorphism $F$ on $S_m'$ as follows: $w_j=z_j$ for $j>1$ and
  $$
  w_1=-\frac1{\bar{z}_1}(|z_2|^2+\cdots+|z_n|^2).
  $$
  Clearly, the vector $F(z)$ is orthogonal to $z$ in ${\mathbb C}^n$ and
  $$
  |F(z)|\le \frac{|z|^2}{|z_1|}< (25c_n)^{-1}R,
  $$
  i.e. $(5c_n)\cdot F(z)\in B(0,R/5)$.
    Since the real Jacobian $J_{\mathbb R}(F)$ of $F$ satisfies
  $$
  J_{\mathbb R}(F)(z)=-\frac1{|z_1|^4}(|z_2|^2+\cdots+|z_n|^2)^2
  $$
  for $z\in S_m'$, it follows that
  $$
  |(5c_n)\cdot F(S_m')|\ge (5 c_n)^5 |S_m'|\ge \frac12 (5 c_n)^5 |S_m|.
  $$
  Thus if we choose
  $$
   m>\frac{1+\frac12 (5 c_n)^5}{\frac12 (5 c_n)^5}(-\psi)_{B(0,R/5)}
  $$
  so that $|S_m|>[1+\frac12 (5 c_n)^5]^{-1} |B(0,R/5)|$, then $S_m\cap ((5c_n)\cdot F(S_m'))\neq \emptyset$. In other words, there exists a complex line $L_1$ such that both $L_1$ and its orthogonal complement $L_1^\bot$ in ${\mathbb C}^n$ intersect $S_m$. Suppose $a\in S_m\cap L^\bot_1$. It follows from the sub mean-value property that
  $$
  \int_{B(a,4R/5)\cap L^\bot_1} |\psi|\le |B(a,4R/5)\cap L^\bot_1| |\psi(a)|.
  $$
  Since $B(0,R/5)\subset B(a,4R/5)$, we have
   $$
  \frac1{|B(0,R/5)\cap L^\bot_1|}\int_{B(0,R/5)\cap L^\bot_1} |\psi|\le C_n m.
  $$
  By repeating the previous argument, we obtain the remaining complex lines $L_2,\cdots,L_n$.
  \end{proof}

  \begin{proof}[Proof of Theorem \ref{th:BMO_PSH}]
  We consider a ball $B(a,R)\subset\subset \Omega$. Replacing $\psi$ by $\psi-C$ where $C$ is a sufficiently large constant, we may assume that $\psi<0$ in a neighborhood of $\overline{B(a,R)}$. We may also assume $a=0$. Let $L_j$, $1\le j\le n$, be chosen as Lemma \ref{lm:BMO_PSH_5}. By a rotation, we may assume that $L_j=\{z:z_1=\cdots =z_{j-1}=z_{j+1}=\cdots=z_n=0\}$. By Lemma \ref{lm:BMO_PSH_3} and Lemma \ref{lm:BMO_PSH_4}, we see that for  $r<R/30$
  $$
   \frac1{|B(0,r)|}\int_{B(0,r)} |\psi-\psi_{B(0,r)}|\le C_n  R^{-2} {|\log R|} \sum_{k=1}^n \int_{|z_k|<R/5}[1+\psi^2(0,\cdots,0,z_k,0,\cdots,0)].
  $$
  Let $b^{(k)}=(0,\cdots,0,b_k,0,\cdots,0)\in L_k\cap S_m$. It follows from (\ref{eq:Hormander}) that
  $$
  \int_{|z_k-b_k|<2R/5}|\psi(0,\cdots,0,z_k,0,\cdots,0)|^2\le C_0 R^{2} |\psi(b^{(k)})|^2\le C_0 m^2 R^{2}.
  $$
  Thus
  $$
   \int_{|z_k|<R/5}|\psi(0,\cdots,0,z_k,0,\cdots,0)|^2\le C_0 m^2 R^{2}.
  $$
  As $m$ depends only on $n$ and $(-\psi)_{B(0,R/5)}$, we conclude that $\psi\in {\rm BMO}(\Omega,{\rm loc})$.
    \end{proof}

  \begin{problem}
   Let $\Omega$ be a domain in ${\mathbb C}^n$. Suppose $\psi\in PSH(\Omega)$ satisfies $e^{-\psi} \in L^1_{\rm loc}(\Omega)$. Is it possible to conclude that for any domain $\Omega'\subset\subset \Omega$ there exists an $\alpha>1$ and a $C>0$ such that
   $$
   \left[\frac1{|B|}\int_B e^{-\alpha \psi}\right]^{1/\alpha}\le \frac{C}{|B|}\int_B e^{-\psi}
   $$
   for all balls $B\subset \Omega'$?
  \end{problem}

 It is also interesting to find conditions for a function $\psi\in PSH({\mathbb C}^n)$ belonging to ${\rm BMO}({\mathbb C}^n)$. A function $\psi\in PSH({\mathbb C}^n)$ is said to be of minimal growth if
 $$
 \psi(z)-\log |z|\le O(1)\ \ \ {\rm as\ \ \ }|z|\rightarrow \infty.
 $$
 Let ${\mathcal L}$ denote the family of such functions.

 \begin{problem}
 Does one has ${\mathcal L}\subset {\rm BMO}({\mathbb C}^n)$?
 \end{problem}

  \section{Liouville properties of subharmonic functions}

 In order to prove Theorem \ref{th:Liouville_1}, we need the following

  \begin{lemma}\label{lm:itegral_Upperbound}
 Let $f\in L^1_{\rm loc}(M)$. Fix a point $x_0\in M$ and set
 $$
 g(r):=\int_{B(x_0,r)}|f|d\mu.
 $$
 Let $0<r<R<\infty$ and $\varepsilon>0$. There exists a locally Lipschitz function $\phi=\phi_{\varepsilon,r,R}$ with  compact support in $\overline{B}(x_0,R)$ such that $\phi=1$ on $\overline{B}(x_0,r)$ and
 \begin{equation}\label{eq:integralUpperbound}
 \int_M¡¡|f| |\nabla \phi|^2 d\mu\le 2 \left[\int_{r}^R \frac{t-r}{g(t)-g(r)+\varepsilon}dt\right]^{-1}.
 \end{equation}
 \end{lemma}

 \begin{proof}
  The argument is parallel to  \cite{Grigoryan}, p.\,36.  Set  $\rho(x)={\rm dist\,}(x_0,x)$. Since $|\nabla \rho|=1$ a.e. on $M$, it follows from the Co-Area formula that
  $$
  g(t)=\int_0^t \int_{S(x_0,s)}|f| d\mu'_s dt
  $$
  where $S(x_0,s)=\partial B(x_0,s)$ and $\mu'_s$ is the corresponding Riemannian measure on $S(x_0,s)$. It follows that
  $$
  g'(t)=\int_{S(x_0,t)}|f| d\mu'_t\ \ \ {\rm a.e}.
  $$
 Let $\chi:{\mathbb R}\rightarrow [0,\infty)$ is a Lipschitz function such that $\chi|_{(-\infty,r)}=1$, $\chi|_{[R,\infty)}=0$ and
  $$
  \chi(t)=c \int_t^R \frac{s-r}{g(s)-g(r)+\varepsilon} ds,\ \ \ t\in [r,R],
  $$
  where
  $$
  1/c:=\int_{r}^R \frac{s-r}{g(s)-g(r)+\varepsilon} ds.
  $$
  Set   $\phi:=\chi\circ \rho$. Then we have
    \begin{eqnarray*}
   \int_M¡¡|f| |\nabla \phi|^2 d\mu & \le & \int_{r}^R \chi'(t)^2 g'(t)dt \le c^2\int_{r}^R \left[  \frac{t-r}{g(t)-g(r)+\varepsilon}\right]^2 d g(t)\\
   & = & -\left.\frac{c^2(t-r)^2}{g(t)-g(r)+\varepsilon}\right|_{r}^R+2c^2 \int_{r}^R \frac{t-r}{g(t)-g(r)+\varepsilon}dt\\
   & \le & 2c.
  \end{eqnarray*}
  \end{proof}

 \begin{proof}[Proof of Theorem \ref{th:Liouville_1}]
  For the case (1), we set $\eta(t)=\int_0^t \frac{ds}{\lambda(s)}$. Let $\phi$ be as in Lemma \ref{lm:itegral_Upperbound}. Assume first that $\psi\in C^2(M)$. Since $1/\eta'=\lambda$, we infer from  (\ref{eq:PoincareIneq_mfd}) that
  \begin{eqnarray*}
  &&  \int_{M} |\phi|^2  \frac{\eta'(-\psi)}{\eta^2(-\psi)}{|\nabla\psi|^2} d\mu  \le  4 \int_M \lambda(|\psi|){|\nabla\phi|^2} d\mu\\
  & \le & 8 \left[\int_{r}^R \frac{t-r}{v_\lambda(t)-v_\lambda(r)+\varepsilon}dt\right]^{-1}\ \ \ ({\rm by\ }(\ref{eq:integralUpperbound}))\\
  & \rightarrow & 0\ \ \ \ \ ({\rm as\ }R\rightarrow \infty).
  \end{eqnarray*}
 Since $\phi=1$ on $B(x_0,r)$, it follows that $\nabla \psi\equiv 0$ on $B(x_0,r)$, hence on $M$ since $r$ can be arbitrarily large. Thus $\psi\equiv {\rm const}$.

 If $\psi$ is only continuous, then for fixed $0<r<R$ there exists a sequence of smooth subharmonic functions $\psi_j$ in a neighborhood of $\overline{B(x_0,R)}$ which converges uniformly to $\psi$ in view of the (local) approximation theorem of Greene-Wu \cite{GreeneWu}. We have
 $$
  \int_{M} |\phi|^2  \frac{\eta'(-\psi_j)}{\eta^2(-\psi_j)}{|\nabla\psi_j |^2} d\mu  \le  4 \int_M \lambda(|\psi_j |){|\nabla\phi |^2} d\mu.
   $$
   Thus for $j\ge j_0(r,R)\gg 1$,
   $$
   \int_{B(x_0,r)} |\nabla\psi_j |^2 \le C \int_M \lambda(|\psi_j |){|\nabla\phi |^2} d\mu
   \le 2C \int_M \lambda(|\psi |){|\nabla\phi |^2} d\mu.
     $$
     where $C$ depends  on $r$, $\eta$ and $\psi$, but independent of $R$. By the Poincar\'e inequality, there exists a positive constant $C_r$ depending only the geometry of $\overline{B(x_0,r)}$ such that
     $$
     \int_{B(x_0,r)} |\psi_j-(\psi_j)_{B(x_0,r)}|^2 \le C_r  \int_{B(x_0,r)} |\nabla\psi_j |^2 \le
     2CC_r \int_M \lambda(|\psi |){|\nabla\phi |^2} d\mu.
           $$
  Letting $j\rightarrow \infty$, we have
  $$
  \int_{B(x_0,r)} |\psi-\psi_{B(x_0,r)}|^2 \le
     2CC_r \int_M \lambda(|\psi |){|\nabla\phi |^2} d\mu\rightarrow 0
 $$
 as $R\rightarrow \infty$. It follows that $\psi\equiv {\rm const.}$ on $B(x_0,r)$ for every $r$.

 For the case (2), we set $
  \eta(t):=\left(\int_t^\infty 1/\lambda\right)^{-1}.
  $
  Since $\eta^2/\eta'=\lambda$, it follows from (\ref{eq:PoincareIneq_mfd+}) and (\ref{eq:integralUpperbound})  that if $\psi$ is $C^2$ then
  $$
  \int_{M}\phi^2  \eta'(\psi){|\nabla\psi|^2} d\mu  \le  4 \int_M \lambda(\psi){|\nabla\phi|^2} d\mu\rightarrow 0
  $$
  as $R\rightarrow \infty$. Hence $\nabla \psi \equiv 0$ and $\psi\equiv {\rm const.}$ The general case follows by a similar argument as above.
 \end{proof}

 We remark that Theorem \ref{th:Liouville_1}/(2) contains as a special case of a theorem of Nadirashvili \cite{Nadirashvli} that the condition $\int_M \frac{\lambda(\psi)d\mu}{1+\rho^2}<\infty$, where $\rho:={\rm dist\,}(x_0,\cdot)$ and $\lambda>0$ is a strictly increasing function with $\int_1^\infty \frac{ds}{\lambda(s)}<\infty$, implies that $\psi$ is a constant. To see this, simply note that
 \begin{eqnarray*}
  +\infty & > & \int_M \frac{\lambda(\psi)d\mu}{1+\rho^2}   \ge  \int_0^r \frac{v_\lambda'(t)dt}{1+t^2}\\
                                     & =  & \left.\frac{v_\lambda(t)}{1+t^2}\right|^r_0-\int_0^r v_\lambda(t)\left[\frac1{1+t^2}\right]'dt \\
                                     & \ge & \frac{v_\lambda(r)}{1+r^2},
 \end{eqnarray*}
  which implies that $\int_{1}^\infty \frac{rdr}{v_\lambda(r)}=\infty$.

 \begin{proof}[Proof of Theorem \ref{th:Liouville_2}]
  Suppose first that $\psi\ge 1$. For all $r>r_0\gg 1$, we have
  \begin{eqnarray*}
   v_\lambda (r) & := &  \int_{B(x_0,r)} \lambda(\psi)d\mu
    =   \int_{\rho<r_0} \lambda(\psi) d\mu+ \int_{r_0\le \rho\le r} \lambda(\psi) d\mu.
     \end{eqnarray*}
 Since $\lambda,\kappa$ are increasing and $\psi\le \kappa\circ \rho$ for $\rho\ge r_0\gg 1$, we have
  \begin{eqnarray*}
   \int_{r_0\le \rho\le r} \lambda(\psi) d\mu \le \int_{r_0}^r \lambda(\kappa(s)) dV_s=\lambda(\kappa(s))V_s|_{r_0}^r -\int_{r_0}^r V_s \lambda(\kappa(s))'ds\le \lambda(\kappa(r))V_r.
  \end{eqnarray*}
   Thus
   \begin{eqnarray*}
  \int_{r_0}^\infty \frac{r dr}{v_\lambda (r)}  \ge  {\rm const.} \int_{r_0}^{\infty} \frac{rdr}{\lambda(\kappa(r))V_r} =\infty,
 \end{eqnarray*}
 so that $\psi$ is a constant in view of Theorem \ref{th:Liouville_1}.

 For general $\psi$, we conclude from  the argument above that there exists a constant $C\ge 1$ with $\max\{\psi,1\}=C$, so that $\psi\le C$. Note that the condition (2) implies that $\int_{1}^\infty \frac{rdr}{V_r}=\infty$, which implies that $M$ is parabolic (cf. \cite{Grigoryan}, Theorem 7.3), so that $\psi$ has to be a constant.
 \end{proof}

 We close this section by posing two questions.

  \begin{problem}
  Let $M$ be a complete Riemannian manifold of finite volume. Is it possible to conclude that a continuous subharmonic function satisfying
        $
  \psi(x)\le o(\rho(x)^2)
  $
as $\rho(x)\rightarrow \infty$ is a constant?
  \end{problem}

  \begin{problem}
  Let $M$ be a complete Riemannian manifold such that $V_r\le {\rm const.} r^2$ for all large $r$. Is it possible to conclude that a continuous subharmonic function satisfying
     $
  \psi(x)\le o(\log \rho(x))
  $
as $\rho(x)\rightarrow \infty$  is a constant?
  \end{problem}

   \section{Proof of Theorem \ref{th:Zariski}}

  We use the method of invariant distances developed by the author in \cite{Chen}. Suppose that $\Omega$ is a bounded domain in ${\mathbb C}^n$ and $G$ is a free, properly discontinuous group of holomorphic automorphims on $\Omega$ such that $M:=\Omega/G$ is a Zariski open subset in a projective algebraic variety $\overline{M}$. By Hironaka's theorem \cite{Hironaka}, we may assume that $S:=\overline{M}\backslash M$ is a divisor with simple normal crossings, which means that for any $a\in S$ there is a coordinate polydisc neighborhood ${\mathbb D}^n$ such that ${\mathbb D}^n\backslash S=({\mathbb D}^\ast)^{l}\times {\mathbb D}^{n-l}$ where ${\mathbb D}^\ast={\mathbb D}\backslash \{0\}$. We call the metric
   $$
   \frac{i}2 \left[\sum_{k=1}^l \frac{dz_k\wedge d\bar{z}_k}{|z_k|^2\log^2|z_k|^2}+\sum_{k=l+1}^n \frac{4dz_k\wedge d\bar{z}_k}{(1-|z_k|^2)^2}\right]
   $$
   the Poincar\'e metric on $({\mathbb D}^\ast)^{l}\times {\mathbb D}^{n-l}$.

   Let $S=\bigcup_{j=1}^N S_j$ where each $S_j$ is non-singular. Let $[S_j]$ be the line bundle on $\overline{M}$ associated to $S_j$ and $\sigma_j$ a holomorphic section of $[S_j]$ generating $S_j$. We may choose a Hermitian metric on $[S_j]$ so that the associated norm $\|\sigma_j\|$ of $\sigma_j$ is less than 1.  Fix a K\"ahler metric ${\omega_0}$ on $\overline{M}$. It is well-known that for $C\gg 1$
  $$
  \omega:=C\omega_0-\sum_{j=1}^N \frac{i}2\partial\bar{\partial}\log(-\log \|\sigma_j\|^2)
  $$
  gives a complete K\"ahler metric of finite volume on $M$, which possesses the singularity of the Poincar\'e metric near every point on $S$ (see e.g. \cite{GriffithsBook}).

  Recall that the Poincar\'e hyperbolic distance of ${\mathbb D}$ is given by
  $$
  d_{{\rm hyp}}(z_1,z_2)= \log \frac{1+\left|\frac{z_1-z_2}{1-\bar{z}_2 z_1}\right|}{1-\left|\frac{z_1-z_2}{1-\bar{z}_2 z_1}\right|},\ \ \ z_1,z_2\in {\mathbb D}.
  $$
  A crucial property is that $\log  d_{{\rm hyp}}$ is a psh function on ${\mathbb D}\times {\mathbb D}$ which is strictly psh on off-diagonal points (cf. \cite{JarnickiPflug}, Proposition 1.18; see also \cite{Chen}, Lemma 8.1). As the Carath\'eodory distance between $z,z'\in \Omega$ is given by
  $$
  c_\Omega(z,z'):=\sup_f\left\{d_{\rm hyp}(f(z),f(z'))\right\}
  $$
  where the supremum is taken over all holomorphic mappings $f:\Omega \rightarrow {\mathbb D}$, it follows that $\log c_\Omega$ is a psh function on $\Omega\times \Omega$.

  Suppose on the contrary that the center of $G$ is nontrivial, i.e. there exists ${\rm id}\neq T_0\in G$ such that $T_0 T=T T_0$ for every $T\in G$. Set $\psi=\log c_\Omega(z,T_0(z)),\,z\in \Omega$. Since $\psi$ is the pull-back of $\log c_\Omega\in PSH(\Omega\times \Omega)$ by the holomorphic map $z\mapsto (z,T_0(z))$, it follows that $\psi\in PSH(\Omega)$. It is also easy to see that $e^\psi$ is a continuous function. We claim that $\psi$ is not a constant. To see this, first take $z_0\in \Omega$ with $T_0(z_0)\neq z_0$ (so that $c_\Omega(z_0,T_0(z_0))>0$) then take a holomorphic mapping $f:\Omega\rightarrow {\mathbb D}$ with $c_\Omega(z_0,T_0(z_0))=d_{\rm hyp}(f(z_0),f(T_0(z_0)))$ (so that $f$ is nonconstant); if $\psi\equiv {\rm const}$, then $\varphi:=\log d_{\rm hyp}(f(z),f(T_0(z)))$ has to be a constant since it attains the maximum at $z_0\in \Omega$, while $\log d_{\rm hyp}$ is strictly psh near $(f(z_0),f(T_0(z_0)))$ and $f$ is nonconstant, which is impossible (compare \cite{Chen}, P.\,1046).

  Since
  $$
  \psi(T(z))=\log c_\Omega(T(z),T_0T(z))=\log c_\Omega(T(z),TT_0(z))=\psi(z)
  $$
  for every $T\in G$, it follows that $\psi$ descends to a nonconstant psh function $\tilde{\psi}$ on $M$ such that $e^{\tilde{\psi}}$ is continuous. Since every psh function is subharmonic w.r.t. any K\"ahler metric on $M$, we obtain from Theorem \ref{th:Liouville_2} (and the subsequent remark) that $\tilde{\psi}\equiv {\rm const.}$ (contradictory) if we can verify
  \begin{equation}\label{eq:Margulis}
   \tilde{\psi}\le {\rm const.}\log (1+\rho_\omega)
  \end{equation}
  where $\rho_\omega$ is the distance (w.r.t. $\omega$) from a fixed point $x_0$.

  Let $k_\Omega$ (resp. $k_M$) be the Kobayashi distance on $\Omega$ (resp. $M$). Since the Kobayashi metric on $M$ is always dominated by the Poincar\'e metric near every point on $S$, it follows that
  $$
  k_M(x_0,\cdot)\le {\rm const.}(1+\rho_\omega).
  $$
    Let $\varpi:\Omega\rightarrow M$ be the natural projection. Let $x,y\in M$ and $z_x\in \varpi^{-1}(x),\,z_y\in \varpi^{-1}(y)$. It is well-known that
  $$
  k_M(x,y)=\inf_{z_y\in \varpi^{-1}(y)} k_\Omega(z_x,z_y)
  $$
   (cf. \cite{KobayashiBook}, p.\,47). Thus for fixed $z_{x_0}\in \varpi^{-1}(x_0)$ and every $z_{x}\in \varpi^{-1}(x)$, we have
  \begin{eqnarray*}
  e^{\tilde{\psi}(x)} & = &  c_\Omega(z_x,T_0(z_x))\le c_\Omega(z_{x_0},z_x)+c_\Omega(z_{x_0},T_0(z_x))\\
                      & = & c_\Omega(z_{x_0},z_x)+c_\Omega(T_0^{-1}(z_{x_0}),z_x)\\
                      & \le & k_\Omega(z_{x_0},z_x)+k_\Omega(T_0^{-1}(z_{x_0}),z_x),
  \end{eqnarray*}
  so that if $\tilde{x}_0:=\varpi(T_0^{-1}(z_{x_0}))$ then
  $$
  e^{\tilde{\psi}(x)}\le k_M({x_0},x)+k_M(\tilde{x}_0,x)\le {\rm const.}(1+\rho_\omega(x)),
  $$
  which yields (\ref{eq:Margulis}).

\section{Appendix}

 1. We first discuss  an application of Hardy-Sobolev type inequalities to singularity theory of real-analytic functions.
 Let $f:U\rightarrow {\mathbb R}$ be a real-analytic function defined in a neighborhood $U$ of $0\in {\mathbb R}^n$ such that $f(0)=0$. Set $Z_f:=f^{-1}(0)$. {\L}ojasiewicz proved that for any compact set $K$ in $U$, there are constants $\alpha\ge 1$ and $C>0$ such that
\begin{equation}\label{eq:Loja_1}
{\rm dist\,}(x,Z_f)^\alpha\le C |f(x)|
\end{equation}
for all $x\in K$ (see e.g. \cite{Loja64}). Moreover, he showed that there are a smaller neighborhood $V$ of $0$ and constants $0<\beta<1$ and $C>0$ such that
\begin{equation}\label{eq:Loja_2}
|f(x)|^\beta\le C |\nabla f(x)|
\end{equation}
for all $x\in V$. The corresponding\/ {\it {\L}ojasiewicz exponent\/} $\alpha_0(f)$ (resp. $\beta_0(f)$) of $f$ at $0$ is the infimum of $\alpha$ (resp. $\beta$) such that (\ref{eq:Loja_1}) (resp. (\ref{eq:Loja_2})) holds in a neighborhood of $0$.

There is another useful quantity called the\/ {\it singularity exponent\/} $c_0(f)$ of $f$ at $0$, which is defined to be the supremum of $c>0$ such that $|f|^{-c}$ is $L^2$ in a neighborhood of $0$.

Let $d_0(f)$ be the codimension of the real-analytic subvariety $Z_f$ at $0$. We have

\begin{theorem}\label{th:Main}
 Let $f:U\rightarrow {\mathbb R}$ be a real-analytic function defined in a neighborhood $U$ of $0\in {\mathbb R}^n$ such that $f(0)=0$.
If \/ $\log f^2$ is subharmonic, then
 \begin{equation}\label{eq:Singular_1}
 c_0(f)\ge \max\left\{ 1-\beta_0(f), 1-\beta_0(f)+\frac{d_0(f)-2}{2\alpha_0(f)}\right\}.
 \end{equation}
 \end{theorem}

\begin{proof}
 Applying ({\ref{eq:PoincareIneq_Log}}) with $\psi=\log f^2$ and $\phi\in C^\infty_0(U)$ with $\phi=1$ in a smaller neighborhood $U'$ of $0$, we obtain
$$
\int_{U'}\frac{|\nabla f|^2}{f^2\log^2 f^2}< \infty.
$$
It follows that for any $\beta>\beta_0(f)$,
$
|f|^{\beta-1}/\log f^2
$
is $L^2$ in suitable neighborhood of $0$,
which implies $c_0(f)\ge 1-\beta_0(f)$.

 Next we set $\delta_{Z_f}:={\rm dist\,}(\cdot,Z_f)$.
  Note that
 $$
 {\rm vol\,}( U'\cap \{\delta_{Z_f}<\varepsilon\})\le {\rm const.}\varepsilon^{d_0(f)},\ \ \  \forall\,\varepsilon>0,
 $$
 holds for some small neighborhood $U'$ of $0$ (see e.g. \cite{Loeser}, Theorem 1.1 and the subsequent remarks). Set $\phi:=\chi\delta_{Z_f}^{-\tau/2}$ for some $\tau>0$, where $\chi\in C_0^\infty(U')$ satisfies $\chi=1$ in a smaller neighborhood $U''$ of $0$. We have
    \begin{eqnarray*}
  \int \phi^2 & \le &    \sum_{k=1}^\infty \int_{U'\cap \{2^{-k-1}<\delta_{Z_f}\le 2^{-k}\}} \delta_{Z_f}^{-\tau}\\
    & \le &  \sum_{k=1}^\infty 2^{\tau (k+1)} {\rm vol\,}(U'\cap \{\delta_{Z_f}\le 2^{-k}\})\\
  & \le & {\rm const.} \sum_{k=1}^\infty 2^{(\tau-d_0(f))k}\\
  & < & \infty
 \end{eqnarray*}
 provided $\tau<d_0(f)$, and similarly,
$$
  \int |\nabla \phi|^2  \le  {\rm const.} \int_{U'} \delta_{Z_f}^{-\tau-2}
    \le  {\rm const.} \sum_{k=1}^\infty 2^{(\tau+2-d_0(f))k}
   <  \infty
$$
 provided $\tau<d_0(f)-2$.

 Applying (\ref{eq:PoincareIneq_Log}) with $\psi=\log f^2$ and $\phi$ as above, we conclude that if $\tau<d_0(f)-2$ then
  $$
\int_{U''} \frac{|\nabla f|^2 \delta_{Z_f}^{-\tau}}{f^2\log^2 f^2} <\infty.
$$
 Let $\alpha>\alpha_0(f)$ and $\beta>\beta_0(f)$. It follows  that $
|f|^{\beta-1-\frac{\tau}{2\alpha}}/\log f^2
$
is $L^2$ in some neighborhood of $0$, so that
$$
c_0(f)\ge 1-\beta+\frac{\tau}{2\alpha}\rightarrow 1-\beta_0(f)+\frac{d_0(f)-2}{2\alpha_0(f)}
$$
as $\alpha\rightarrow \alpha_0(f)$, $\beta\rightarrow \beta_0(f)$ and $\tau\rightarrow d_0(f)-2$.
\end{proof}

2. Next we give an application of Hardy-Sobolev type inequalities to the strong unique continuation (SUC) property for Schr\"odinger operators. Carleman first proved the following

\begin{theorem}[cf. \cite{Carleman33}, \cite{Carleman}]\label{th:Carleman}
Let $U$ be a domain in ${\mathbb R}^2$ with $0\in U$. Suppose $u\in W^{2,2}_{\rm loc}(U)$ satisfies
 $$
 |\Delta u| =O(|u|+|\nabla u|).
 $$
Then $u$ vanishes identically if it has a zero of infinite order at\/ $0$, i.e.
$$
\lim_{\varepsilon\rightarrow 0} \varepsilon^{-N}\int_{B_\varepsilon} |u|^2=0,\ \ \ \forall\,N\in {\mathbb Z}^+
$$
where $B_\varepsilon=\{x:|x|<\varepsilon\}$.
\end{theorem}

In order to prove this result, Carleman introduced a method, the so-called Carleman estimates, which is essential in almost all the subsequent work on the subject.  Theorem \ref{th:Carleman}  was extended to elliptic equations of many independent variables by various authors (see \cite{HormanderPDE}, \cite{Wolff} and the references therein).
 The following $L^2$ Carleman estimates turn out to be of particular importance:
    \begin{equation}\label{eq:Carleman_Known_1}
  \left\||x|^{-\tau} \phi\right\|_2\le C_1 \tau^{-1}\left\||x|^{2-\tau} \Delta \phi\right\|_2
  \end{equation}
   \begin{equation}\label{eq:Carleman_Known_2}
  \left\||x|^{1-\tau} \nabla \phi\right\|_2\le C_2 \left\||x|^{2-\tau} \Delta \phi\right\|_2
  \end{equation}
  for all $\phi\in C^\infty_0({\mathbb R}^n\backslash \{0\})$ and $\tau\ge n$ with ${\rm dist\,}(\tau-n/2,{\mathbb Z})>0$ (see e.g. \cite{PanWolff}).
  It may be of some interest that (\ref{eq:Carleman_Known_2}) is essentially a formal consequence of (\ref{eq:Carleman_Known_1}).  If one chooses $\eta(t)=t$ in $(\ref{eq:Laplace+})$, then
 \begin{equation}\label{eq:CacciopoliIneq}
\int_{{\mathbb R}^n} \phi^2  {|\nabla\psi|^2} \le 4 \int_{{\mathbb R}^n} \psi^2 |\nabla\phi|^2  - 2\int_{{\mathbb R}^n}\phi^2 \psi \Delta\psi.
 \end{equation}
  With $\psi$ and $\phi$  replaced by $\phi$ and $|x|^{1-\tau}$ respectively in (\ref{eq:CacciopoliIneq}), we obtain
  \begin{eqnarray*}
     \left\||x|^{1-\tau} \nabla \phi\right\|_2 & \le & 4(\tau-1)^2 \left\||x|^{-\tau} \phi\right\|_2^2-2\int_{{\mathbb R}^n} |x|^{2-2\tau}\phi \Delta \phi\\
     & \le & 4(\tau-1)^2 \left\||x|^{-\tau} \phi\right\|_2^2+2\left\||x|^{-\tau} \phi\right\|_{2}\left\||x|^{2-\tau} \Delta \phi\right\|_2\\
     & \le & (4C_1^2 (\tau-1)^2\tau^{-2}+2C_1 \tau^{-1})\left\||x|^{2-\tau} \Delta \phi\right\|_2^2
    \end{eqnarray*}
    in view of (\ref{eq:Carleman_Known_1}).

  A useful consequence of (\ref{eq:Carleman_Known_1}) and (\ref{eq:Carleman_Known_2}) is given as follows. Since
  $$
  \nabla (|x|^{1-\tau} \phi)=|x|^{1-\tau} \nabla \phi+(1-\tau) \phi |x|^{-\tau} \nabla |x|,
  $$
  it follows that
  \begin{eqnarray}\label{eq:Carleman_Gradient}
  \left\|\nabla (|x|^{1-\tau} \phi) \right\|_2 & \le &  \left\||x|^{1-\tau} \nabla \phi\right\|_2+(\tau-1)\left\||x|^{-\tau} \phi\right\|_2\nonumber\\
  & \le &  C \left\||x|^{2-\tau} \Delta \phi\right\|_2.
  \end{eqnarray}

    \begin{theorem}\label{th:SUC_Laplace}
  Let $U$ be a bounded neighborhood of\/ $0\in {\mathbb R}^n$ with $n>2$. Suppose $\omega,\omega'$ are measurable and positive a.e. on $U$ such that for some $\alpha\ge 2$,
  \begin{equation}\label{eq:Hardy-Sobolev_6}
 \left[\int_{U} |\phi|^\alpha \omega\right]^{2/\alpha}\le C \int_{U} |\nabla \phi|^2,\ \ \ \forall\,\phi\in C_0^\infty(U)
 \end{equation}
        \begin{equation}\label{eq:Carleman_Condition_2}
    |x|\omega'\in L^{2\alpha/(\alpha-2)}_{\rm loc}(U).
    \end{equation}
         If $u\in W^{2,2}_{\rm loc}(U)$ has a zero of infinite order at\/ $0$ and satisfies
 \begin{equation}\label{eq:SUC_Laplace_0}
 |\Delta u|\le \omega^{1/\alpha}\omega' |u|,
 \end{equation}
   then $u\equiv 0$ in a neighborhood of\/ $0$.
  \end{theorem}

   \begin{proof}
  Consider a ball $B_{r_0}\subset\subset U$.  By  (\ref{eq:Carleman_Gradient}) and (\ref{eq:Hardy-Sobolev_6}), we have
    \begin{equation}\label{eq:Carleman_Laplace_1}
  \|\omega^{1/\alpha}|x|^{1-\tau} \phi\|_\alpha
  \le  C \left\||x|^{2-\tau} \Delta \phi\right\|_2
  \end{equation}
  for all $\phi\in C^\infty_0(B_{r_0}\backslash \{0\})$.
  Clearly, the same inequality holds for all $\phi\in W^{2,2}_0(B_{r_0}\backslash \overline{B}_\varepsilon)$ where $\varepsilon< r_0$.
  Now we follow the classical argument of Carleman. Fix $r<r_0$ for a moment. For $\varepsilon<r/2$, we choose $\chi_\varepsilon\in C^\infty_0(B_{r})$ such that
  \begin{enumerate}
  \item $\chi_\varepsilon=0$ on $B_{\varepsilon/2}$;
  \item $|\nabla \chi_\varepsilon|\le C\varepsilon^{-1}$, $|\Delta \chi_\varepsilon|\le C \varepsilon^{-2}$ on $B_\varepsilon\backslash B_{\varepsilon/2}$;
  \item $
  \chi_\varepsilon=1
  $
  on ${B_{r/2}\backslash B_{\varepsilon}}$;
\item $|\nabla \chi_\varepsilon|\le C$, $|\Delta \chi_\varepsilon|\le C $ on $B_{r}\backslash B_{r/2}$.
  \end{enumerate}
  A computation gives  $\Delta (\chi_\varepsilon u)=u\Delta\chi_\varepsilon+2\nabla \chi_\varepsilon\cdot \nabla u+\chi_\varepsilon \Delta u$.
  Substituting $\phi=\chi_\varepsilon u$ in to (\ref{eq:Carleman_Laplace_1}), we have
  \begin{eqnarray}\label{eq:Carleman_Laplace_3}
   \|\omega^{1/\alpha}|x|^{1-\tau} u\|_{L^\alpha({B_{r/2}\backslash B_{\varepsilon}})}   & \le & C\varepsilon^{-2} \||x|^{2-\tau} u\|_{L^2({B_{\varepsilon}\backslash B_{\varepsilon/2}})}+ C\varepsilon^{-1}  \||x|^{2-\tau} \nabla u\|_{L^2({B_{\varepsilon}\backslash B_{\varepsilon/2}})} \nonumber\\
   && + C  \||x|^{2-\tau} \Delta u\|_{L^2({B_{r/2}\backslash B_{\varepsilon/2}})}\nonumber\\
   &&  + C \||x|^{2-\tau} (|u|+|\nabla u|+|\Delta u|)\|_{L^2(B_{r}\backslash B_{r/2})}.
  \end{eqnarray}
  Since $u$  has a zero of infinite order at $0$, it follows that for $\tau\ge 2$
  $$
  \int_{B_\varepsilon\backslash B_{\varepsilon/2}} \frac{u^2}{|x|^{2\tau-4}}\le (\varepsilon/2)^{4-2\tau}\int_{B_\varepsilon\backslash B_{\varepsilon/2}} {u^2}=o(\varepsilon^4)
  $$
  and
  $$
  \int_{B_\varepsilon\backslash B_{\varepsilon/2}}  \frac{|\nabla u|^2}{|x|^{2\tau-4}}  \le  (\varepsilon/2)^{4-2\tau}\int_{B_\varepsilon\backslash B_{\varepsilon/2}} {|\nabla u|^2}\\
     =  o(\varepsilon^2)
 $$
  since
  $$
  \int_{B_\varepsilon\backslash B_{\varepsilon/2}} {|\nabla u|^2}\le C\varepsilon^{-2}  \int_{B_{2\varepsilon}} u^2 +2\|u\|_{L^2(B_{2\varepsilon})}\|\Delta u\|_{L^2(B_{2\varepsilon})}
  $$
  in view of (\ref{eq:CacciopoliIneq}). We also have
  \begin{eqnarray*}
  \||x|^{2-\tau} \Delta u\|_{L^2({B_{r/2}\backslash B_{\varepsilon/2}})} & \le &  \|\omega^{1/\alpha}\omega' |x|^{2-\tau}  u\|_{L^2({B_{r/2}\backslash B_{\varepsilon/2}})}\\
  & \le &   \|\omega'|x|\|_{L^{\frac{2\alpha}{(\alpha-2)}}({B_{r/2}\backslash B_{\varepsilon/2}})}\|\omega^{1/\alpha}|x|^{1-\tau} u\|_{L^\alpha({B_{r/2}\backslash B_{\varepsilon/2}})}
   \end{eqnarray*}
   in view of H\"older's inequality.
  Letting $r\ll1$ and $\varepsilon\rightarrow 0$ in (\ref{eq:Carleman_Laplace_3}), we obtain
  \begin{eqnarray*}
  \|\omega^{1/\alpha}|x|^{1-\tau} u\|_{L^\alpha(B_{r/2})} \le C(r/2)^{2-\tau} \|u\|_{W^{2,2}(B_r)},
  \end{eqnarray*}
  so that $ \| \omega^{1/\alpha} (2|x|/r)^{1-\tau} u\|_{L^\alpha(B_{r/2})}$ is uniformly bounded in $\tau$. Thus $u\equiv 0$ in $B_{r/2}$.
    \end{proof}

\end{document}